\documentclass[11pt]{article}
\usepackage[a4paper]{geometry}
\usepackage{amsthm}
\usepackage{amsmath}
\usepackage{amssymb}
\usepackage{amsfonts}
\usepackage{hyperref}
\usepackage{epsfig,graphicx}
\usepackage{xcolor}
\usepackage{subcaption}
\usepackage{authblk}
\usepackage[]{algorithm2e}
\usepackage[normalem]{ulem}

\begin{document}

\newtheorem{Integrator}{Integrator}

\theoremstyle{definition}
\newtheorem{Theorem}{Theorem}[section]
\newtheorem{Definition}{Definition}[section]
\newtheorem{Remark}{Remark}[section]
\newtheorem{Example}{Example}[section]
\newtheorem{Condition}{Condition}

\title{Accurate and Efficient Simulations of Hamiltonian Mechanical Systems with Discontinuous Potentials}

\author[1]{Molei Tao}
\author[2]{Shi Jin}
\affil[1]{School of Mathematics, Georgia Institute of Technology, Atlanta GA 30332, USA. Email: \url{mtao@gatech.edu}}
\affil[2]{School of Mathematical Sciences, Institute of Natural Sciences and MOE-LSE, Shanghai Jiao Tong University, Shanghai 200240, China.  Email: \url{shijin-m@sjtu.edu.cn}}

\maketitle

\centerline{\it  Dedicated to the centenary of the birth of Kang Feng}

\begin{abstract}
This article considers Hamiltonian mechanical systems with potential functions admitting jump discontinuities. The focus is on  accurate and efficient numerical approximations of their solutions, which will be defined via the laws of reflection and refraction. Despite of the success of symplectic integrators for smooth mechanical systems, their construction for the discontinuous ones is nontrivial, and numerical convergence order can be impaired too. Several rather-usable 
numerical methods are proposed, including: a first-order symplectic integrator for general problems, a third-order symplectic integrator for problems with only one linear interface, arbitrarily high-order reversible integrators for general problems (no longer symplectic), and an adaptive time-stepping version of the previous high-order method. Interestingly, whether symplecticity leads to favorable long time performance is no longer clear due to discontinuity, as traditional Hamiltonian backward error analysis does not apply any more. Therefore, at this stage, our recommended default method is the last one. Various numerical evidence, on the order of convergence, long time performance, momentum map conservation, and consistency with the computationally-expensive penalty method, are supplied. A complex problem, namely the Sauteed Mushroom, is also proposed and numerically investigated, for which multiple bifurcations between trapped and ergodic dynamics are observed.
\end{abstract}

\section{Introduction}
The developments of accurate and efficient integrators for simulating smooth Hamiltonian mechanical systems, as well as the associated theoretical analysis,  have been a major triumph of contemporary numerical analysis (see e.g.,  \cite{Hairer06, calvo1994numerical, feng2010symplectic, leimkuhler2004simulating, blanes2017concise}). Symplectic integrators (e.g., \cite{feng1986difference,sanz1992symplectic}), for instance, are a celebrated class of numerical methods suitable for such systems. For example, explicit symplectic integrators have been constructed, with arbitrarily high-order versions for both separable Hamiltonians (e.g., \cite{creutz1989higher, forest1989canonical, suzuki1990fractal, Yoshida:90}) and general, non-separable Hamiltonians \cite{Tao2016PRE}. Their explicitness and the ability to use relatively large timesteps lead to computationally efficient simulations (for stiff/multiscale problems, see also, e.g., \cite{Skeel:99, Sanz-Serna:08, SIM2, FLAVOR10}), and the high-orderness, which has mostly been achieved by a powerful technique known as the splitting method (e.g., \cite{McQu02} for a review), yields high accuracy at least for short-time simulations. Moreover, favorable long-time properties of symplectic integrators have also been proved, including linear growth of error (for integrable systems, e.g., \cite{calvo1995accurate, quispel1998volume}), near preservation of energy (e.g., \cite{benettin1994hamiltonian}), and conservation of momentum maps associated with symmetries (e.g., \cite{MaWe:01}). Central to many of these beautiful analyses is what is nowadays known as (Hamiltonian) backward error analysis (e.g., \cite{Hairer06}). It views the iterations of a symplectic integrator as stroboscopic samples of the solution of a near-by Hamiltonian system, which is to be found and hopefully close to the original Hamiltonian. In addition, also worth noting is another class of useful methods for structured continuous systems, namely reversible integrators. This class often overlaps with symplectic integrators, although they are not always the same, and for them one can also establish long term accuracy via (reversible) backward error analysis under reasonable assumptions (see Chap.XI of \cite{Hairer06}).

On the other hand, if the potential of a mechanical system has discontinuity, each corresponding to a potential barrier\footnote{Here we assume one is simply interested in a Newtonian problem (i.e., $H(q,p)=\|p\|^2/2+V(q)$), which is a special case of separable Hamiltonian problems defined via $H(q,p)=K(p)+V(q)$, where $K$ and $V$ are respectively referred to as the kinetic and the potential energy).}, most aforementioned results no longer hold. Even the sense in which one discusses the solution has to be defined, because the standard equations of motion known as the (canonical) Hamilton's equation (i.e. $\dot{q}=\partial H/\partial  p, \dot{p}=-\partial H/\partial  q$) is ill-defined due to indifferentiability of $V(q)$. Following \cite{Jin-Wen1, JinReview} (and also 
its higher-dimensional extension to curved interfaces \cite{Jin-Liao},  for Hamiltonian systems with discontinuous Hamiltonians \cite{JinWuHuang}, as well as an earlier such approach for well-balanced schemes for the shallow-water equations \cite{PerthameSimeoni}), we will define a solution based on physical principle, or more precisely, using the classical idea of particle refraction and reflection (used in optics and derived rigorously from Maxwell's equation \cite{jackson2007classical}). See Fig. \ref{fig_impactIllustration1D} for a simplified illustration and Fig. \ref{fig_impactIllustration} for a more general case.
This definition will be numerically shown (Sec.\ref{sec_numericsBenchmark} and \ref{sec_numericsModifiedKepler}) to be consistent with an alternative treatment termed as the penalty method (see Sec.\ref{sec:relatedWork} for a brief discussion of the penalty method), which was also observed in \cite{Jin-Wen2}, but with improved efficiency and accuracy.

\begin{figure}[h!]
\center
\includegraphics[width=0.5\textwidth]{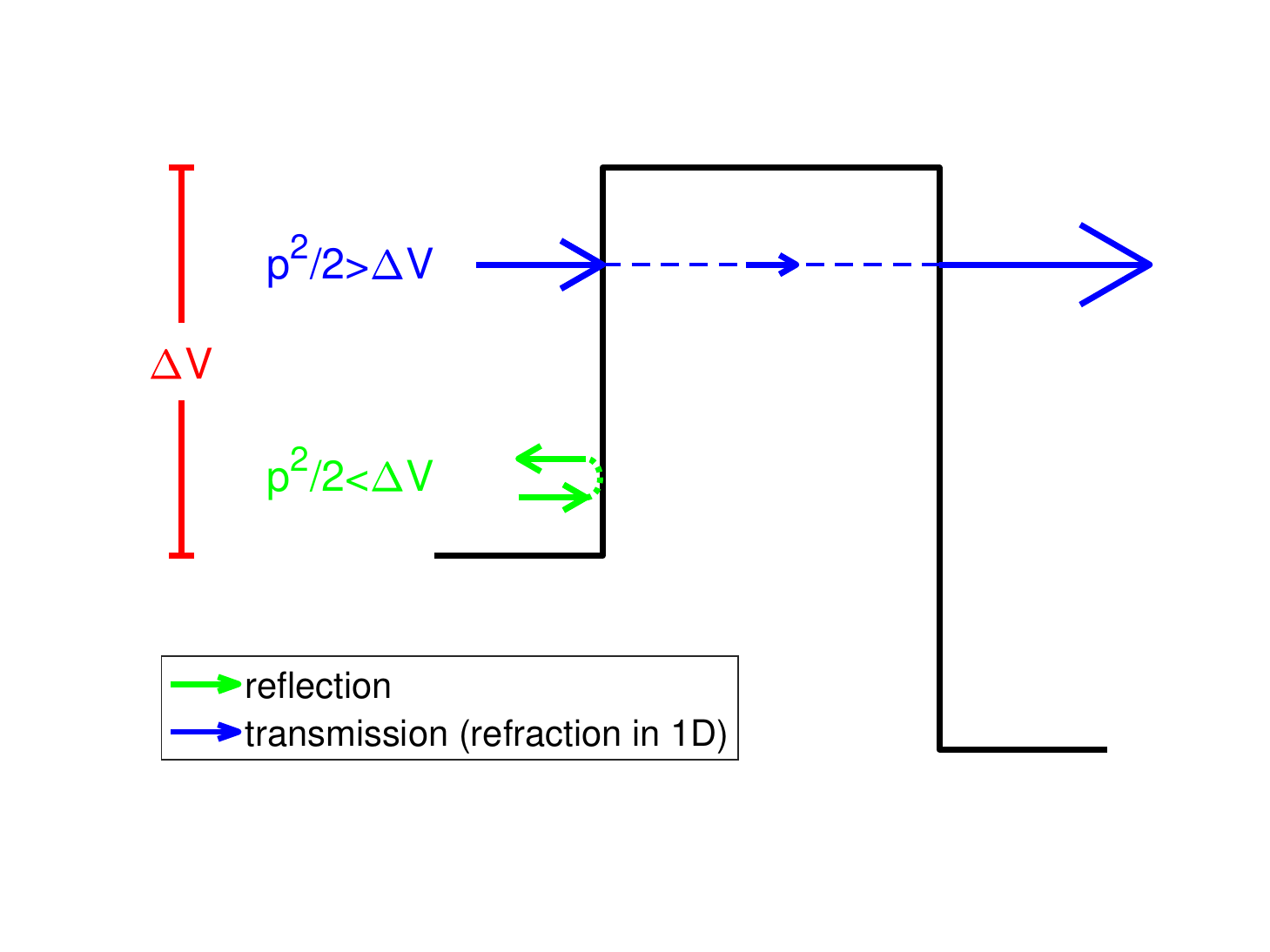}
\vspace{-20pt}
\caption{\small Toy 1D illustration of the dichotomy of \emph{reflection} and \emph{refraction}: whether the particle goes through a discontinuous barrier depends on if it has enough kinetic energy to overcome it. In 1D, this is a consequence of energy conservation. Arrow size indicates the magnitude of momentum.}
\label{fig_impactIllustration1D}
\end{figure}

In the previous works \cite{Jin-Wen1, JinWuHuang}, the proposed schemes were in general neither symplectic nor of high order accuracy. 
 In fact, since the so-defined  solution will exhibit discontinuity, in particular in the momentum variable, as a response to the discontinuity in the potential, these two discontinuities make the construction of a symplectic integrator, or even just a reversible integrator, nontrivial. Designing high-order methods becomes even more challenging, as the splitting approach for boosting the convergence order will be shown no longer effective due to the discontinuity. Moreover, backward error analysis, either the one for symplectic integrators or the one for reversible integrators, fails as well, and long time performance guarantees are not proved any more.

In order to improve the numerical simulation for such singular Hamiltonian systems, we propose four numerical methods, each specializing in certain tasks. See Table \ref{tab_methods}. Properties of these methods are numerically studied in Sec. \ref{sec_numerics}, such as convergence order (Sec. \ref{sec_numericsBenchmark}\&\ref{sec_numerics3rdOrder}), long time accuracy (Sec. \ref{sec_numericsBenchmark},\ref{sec_numerics3rdOrder},\ref{sec_numericsModifiedKepler}), and the conservation of momentum map due to symmetry (Sec. \ref{sec_numericsModifiedKepler}).

\begin{table}[h!]
{\centering
 \begin{tabular}{c|c|c|c|c} 
 \hline
	\textbf{section} & \textbf{symplectic?} & \textbf{reversible?$^\#$} & \textbf{global error$^*$} & \textbf{other feature(s)} \\
 \hline
	\ref{sec_1stOrderSympl}	& Yes	& Yes	& 1st-order & general \\
 \hline
	\ref{sec_3rdOrderSympl}	& Yes	& Yes	& 3rd-order & 1 linear interface$^\dagger$ only \\
 \hline
	\ref{sec_highOrderNonSympl} & No & Yes	& arbitrary & general \\
 \hline
	\ref{sec_highOrderNonSymplAdaptive} & No & Yes	& arbitrary & general; adaptive time-stepping \\
 \hline
 \end{tabular}
 }

 \#: the more precise question is whether the integrator can be made reversible.\\
 *: global error considered here is that of position, away from interface interceptions. \\
 $\dagger$: a linear interface is a co-dimension 1 hyperplane of discontinuities in the potential.
 \caption{A brief summary of numerical methods proposed in this article.}
 \label{tab_methods}
\end{table}


\subsection{Related work}
\label{sec:relatedWork}

\paragraph{`Nonsmooth mechanics'.} A rich field termed `nonsmooth mechanics' / `nonsmooth Hamiltonian systems' / `mechanical systems with hard constraints' / `unilaterally constrained systems' / `contact integrators' / `collision integrators' already exists (e.g., \cite{stratt1981constrained, mcneil1982new, heyes1982molecular, suh1990molecular, laursen1997design, kane1999finite, houndonougbo2000molecular, stewart2000rigid, laursen2002improved, pandolfi2002time, fetecau2003nonsmooth, cirak2005decomposition, bond2008stabilized, deuflhard2008contact, khenous2008mass, doyen2011time, pekarek2012variational, krause2012presentation, kaufman2012geometric, dharmaraja2012time,  leyendecker2012variational}). Problems considered there correspond to a special case of this study. 

More precisely, to the best of our knowledge, these  literature mainly consider, in the language of this article, a discontinuous potential barrier of height $+\infty$, so that trajectories can only stay within the finite potential region\footnote{Note these literature equivalently formulated the problem as there are a collection of unilateral holonomic constraints $f_i(q) \geq 0$.}. This setup is already very important in engineering and science applications; for example, in robotics, the interested mechanical object is often interacting with hard surfaces -- think about a bipedal robot that walks by frequently impacting the ground (e.g., \cite{grizzle2014models}), and in molecular and polymer dynamics molecules are sometimes viewed as hard spheres (so that they remain at least a certain distance from each other; e.g., \cite{dijkstra2002phase}). However, in these cases, the boundary manifold (i.e., the interface(s)) cannot be crossed, and one will only have \emph{reflection}s but never \emph{refraction}s. Therefore, those interactions with the interface were commonly called in the literature `contacts' and `collisions'.

On the contrary, the setup in \cite{Jin-Wen1, JinWuHuang}, which is adopted in this article,  allows finite discontinuous jumps, and therefore the full dichotomy of \emph{reflection} \textbf{and} \emph{refraction} can both manifest in the dynamics. 

In addition to this main difference, here are some more details on other aspects in which this research compares with the existing works in nonsmooth mechanics (all for $+\infty$ jump only): One objective of this article is to develop explicit symplectic integrators. Some pioneering breakthroughs developed symplectic integrators which are however implicit (e.g., \cite{fetecau2003nonsmooth, leyendecker2012variational}). An intriguing paper \cite{kaufman2012geometric} also noted that if one relaxes the symplecticity requirement and instead only requires symplecticity over smooth trajectories intervals, then it is possible to obtain better energy behaviors. This also reminds us of an inspiring and clever earlier work \cite{bond2008stabilized}, which uses backward error analysis for continuous systems to design stabilized event-driven collision integrator. In addition, there is a collection of substantial works based on finite element (e.g., \cite{kane1999finite, khenous2008mass, cirak2005decomposition, doyen2011time}). Although inevitably incapable of discussing all results in the rich `$+\infty$ jump' field, we also mention the relevant paper \cite{dharmaraja2012time} in the context of symplectic integrator. Its main goal is to simulate smooth potentials, but it approximates the smooth potential by a piecewise constant or quadratic function and uses analytically obtained solution of the nonsmooth approximation as a numerical solution.  What is new in this paper, in comparison, is we will have a continuous part of the potential in addition to this piecewisely-defined discontinuous part.

\paragraph{Penalty method / regularization.}
A popular idea commonly known as regularization/regularisation corresponds to modifying the discontinuous vector field of a differential equation and replacing its discontinuities by steep but continuous transitions. After the regularization both the equations of motion and the solution become well-defined. The hope is to recover the original dynamics as the steepness goes to infinity. This idea was proved to be very useful, for example, in engineering applications (e.g, \cite{gonthier2004regularized}). 
It should be pointed out, however, that regularization generally creates artificial numerical stiffness (see two paragraphs below), and caution should be exercised even without considering numerics, as regularization doesn't always guarantee a good, or even correct, approximation. In fact, general discontinuous problems may not even have unique solutions (e.g., \cite{filippov2013differential,dieci2009sliding}), and \cite{nordmark2011friction}, for example, provides a mathematical discussion on when regularization actually removes this ambiguity (we also refer to the notion of renormalized solution \cite{DiPernaLions, Ambrosio}, which is another way of removing ambiguity, \textbf{not} via regularization though). 
Also, regularization isn't always possible (e.g., \cite{makarenkov2012dynamics}). Furthermore, in
the case of geometric optics through an interface, an arbitrary smoothing of the interface could lead to incorrect (partial) transmission
and reflection rates \cite{Jin-Wen3}. Profound analyses exist and provide sufficient conditions for effective regularization (e.g., \cite{fridman2002slow, llibre2008sliding, llibre2007regularization, teixeira2012regularization, makarenkov2012dynamics}), however only for several subclasses of problems. 

The setup in this article is also just a subclass, because we only consider Newtonian mechanical systems and only the potential is discontinuous. In some sense this produces a higher-order singularity in the problem than what the previous paragraph discussed, as the forcing term in the `vector field' will become not just discontinuous but Dirac. It is natural to interpret regularization in this case to be the (sufficiently differentiable) regularization of the potential function instead of the vector field. This idea appeared, for example, in the `softening' of gravitational potential (e.g., \cite{aarseth1963dynamical}). Again, such regularization can work very well for specific scientific investigations, but its general validity is not warranted (e.g, \cite{gerber1996global} for an empirical example).

Focusing directly on the discontinuous problem, this article bares no ambition of investigating the general validity of regularized potentials, but only uses their numerical simulations as one of few available methods to compare to. We will simulate the regularized Hamiltonians using classical smooth symplectic integrators and use the numerical solutions as approximations of the discontinuous solutions. This approach will be called \emph{the penalty method} thereafter. The specific form of regularization used in our experiments is based on sigmoid function (e.g., eq.\ref{eq_regularizedHamiltonianNumerics1},\ref{eq_regularizedHamiltonianNumerics3}), and for these examples, the regularized dynamics do appear to be a good approximation of our definition of the exact (discontinuous) solution. We conjecture the accuracy of this approximation to be $\mathcal{O}(1/\alpha)+\mathcal{O}((\alpha h)^p)$, away from time points of interface crossing, if generic integrators are used (for some non-generic stiff integrators, see e.g, \cite{MR1436164,tao2016variational}, but their error bounds in this case are unclear yet). Here $\alpha$ is the steepest parameter that should go to $\infty$, and $h$ and $p$ are respectively the step size and order of the smooth integrator used in the penalty method. One can see that the artificial stiffness created by the regularization poses a strong constraint on the step size, which renders the penalty method computationally inefficient; such severe time-step constraints were already known (e.g., \cite{Jin-Wen2}). Our proposed discontinuous integrators, on the other hand, do not have this restriction.


\section{The definition of an exact solution}
\subsection{The setup of the problem}
Consider a Newtonian mechanical system in a $2d$-dimensional Euclidean phase space, whose potential function has jump discontinuities across interfaces. Assuming the location and size of each discontinuity is known, then the potential can be decomposed as the sum of a continuous part and a piecewise constant part. Assume also that the continuous part is two-times continuously differentiable. That is, formally, the system is governed by a Hamiltonian\footnote{Note the mass matrix has been assumed to be $I$, because other mass matrices can be equivalently turned into the identity via coordinate changes, which simply correspond to alternative $V$ and $U$; see, e.g., \cite{SoTa18} for a summary of how to transform the potentials.}
\begin{equation}
	H(q,p)=\frac{1}{2}p^T p+V(q)+U(q),
	\label{eq_canonicalHamiltonian}
\end{equation}
where $U(\cdot)$ is a $\mathcal{C}^2$ function, and $V(q)=V_i$ for some constant $V_i$ when $q\in D_i$. $D_i$'s for $i=1,\cdots,M$ are open sets whose closure form a partition of the configuration space $\mathbb{R}^d$. Let
\[
B_{ij}=\begin{cases}
		\overline{D_i} \cap \overline{D_j}, &\qquad i\neq j \\
		\emptyset,	&\qquad i=j
\end{cases} \]
denote the discontinuity interfaces and assume they are either empty or 1-codimensional $\mathcal{C}^1$ submanifolds.

\subsection{Exact solution via physical laws of reflection and refraction}
\label{sec_exactSln}
Due to the non-differentiability of $V(\cdot)$, Hamilton's equation can no longer be used to describe the (meta)particle's global motion. Nevertheless, one can turn to mechanical behavior of particles at the interface to define the solution, as proposed in 
\cite{Jin-Wen1, JinWuHuang, JinReview} (for curved interface in high dimension see \cite{Jin-Liao}): basically, in order for the solution to make sense physically, a corresponding particle should simply evolve locally according to the smooth Hamiltonian dynamics given by $\hat{H}=\frac{1}{2}p^T p+U(q)$, until it hits an interface, and then the particle will either reflect or refract instantaneously, depending on the normal momentum magnitude and whether the jump in $V$ corresponds to a potential barrier or dip across the interface. Then the particle evolves again locally in some $D_i$ according to $\hat{H}$, until the next interface hitting.

More precisely, under nontrivial but not too restrictive assumptions (see Conditions \ref{asmpt_oneInterface} and \ref{asmpt_noInterfaceSliding}), the solution will be well-defined as an alternation between two phases, \emph{flow} and \emph{impact}, which will now be detailed.
To do so, denote by $\phi^t$ the time-$t$-flow map of $\hat{H}$, and by $\mathcal{Q}$ the operator that projects $[q,p]$ to the $q$-component. Let $t_0$ be the initial time of evolution, and let $i_{t_0}$ be the integer such that the initial condition satisfies $q(t_0)\in \overline{D_{i_{t_0}}}$ (note: if the initial condition is on an interface, $i_{t_0}$ is not unique, and its choice needs to be specified as part of the initial condition). The next hitting time is defined to be
\[
	t_{k+1}=t_k+\min_{j=1,\cdots,M} \inf \left\{ \delta \,\Big\rvert\, \delta > 0, \mathcal{Q}\circ\phi^\delta[q(t_k),p(t_k)] \in B_{i_{t_0}j} \right\},
\]
let
\[
	i_{t_{k+1}}=\text{argmin}_{j=1,\cdots,M} \inf \left\{ \delta \,\Big\rvert\, \delta > 0, \mathcal{Q}\circ\phi^\delta[q(t_k),p(t_k)] \in B_{i_{t_0}j} \right\},
\]
and let the solution on time interval $[t_k, t_{k+1}]$ be defined by
\[
	[q(t_k+\delta),p(t_k+\delta)]=\phi^\delta[q(t_k),p(t_k)],	\qquad 0 \leq \delta \leq t_{k+1}-t_k.
\]
This gives one of the two phases of the solution, which shall be called \emph{flow}.

\begin{Remark}
	It is easy to see $\emph{flow}$ preserves both (i) the differentiable part of the total energy, $\hat{H}$, and (ii) the total energy $H$ as long as time points $t_k$ and $t_{k+1}$ are not considered (on which $V$ is ill-defined).
\end{Remark}

After \emph{flow}, the other phase, called \emph{impact}, will take place (unless $t_{k+1}=\infty$).

To define \emph{impact}, let $\hat{n}$ be the unit normal vector of the interface $B_{i_{t_k}i_{t_{k+1}}}$ at $q(t_{k+1})$, in the direction of from $D_{i_{t_k}}$ to $D_{i_{t_{k+1}}}$. Decompose the pre-impact momentum as $p(t_{k+1})=p^t_{k+1}+p^n_{k+1}$, where $p^n_{k+1}=(p(t_{k+1})\cdot \hat{n})\hat{n}$ is the projection onto the normal direction. Let $\Delta V=V_{t_{k+1}}-V_{t_k}$.

If the particle has enough normal momentum to transmit through the interface, i.e.,
\[
	\| p^n_{k+1} \|^2/2 \geq \Delta V,
\]
a \emph{refraction} will happen in the sense that the post-impact normal momentum will be reduced, and the value of $p(t_{k+1})$ will be overwritten by conservation of energy and tangential momentum:
\begin{equation}
	p(t_{k+1}) = p^t_{k+1} + \sqrt{\|p_{k+1}^n\|^2-2\Delta V} \hat{n}.
	\label{eq_refraction}
\end{equation}
If there is not enough normal momentum on the other hand, i.e.,
\[
	\| p^n_{k+1} \|^2/2 < \Delta V,
\]
a \emph{reflection} will take place in the sense that the post-impact normal momentum will simply change its sign, and the value of $p(t_{k+1})$ will be overwritten by
\begin{equation}
	p(t_{k+1}) = p^t_{k+1} - p^n_{k+1}.
	\label{eq_reflection}
\end{equation}
However, in both cases, the value of the position, $q(t_{k+1})$, will remain unchanged. An illustration of \emph{refraction} and \emph{reflection} is provided in Figure \ref{fig_impactIllustration}.

\begin{figure}[h!]
\center
\includegraphics[width=0.5\textwidth]{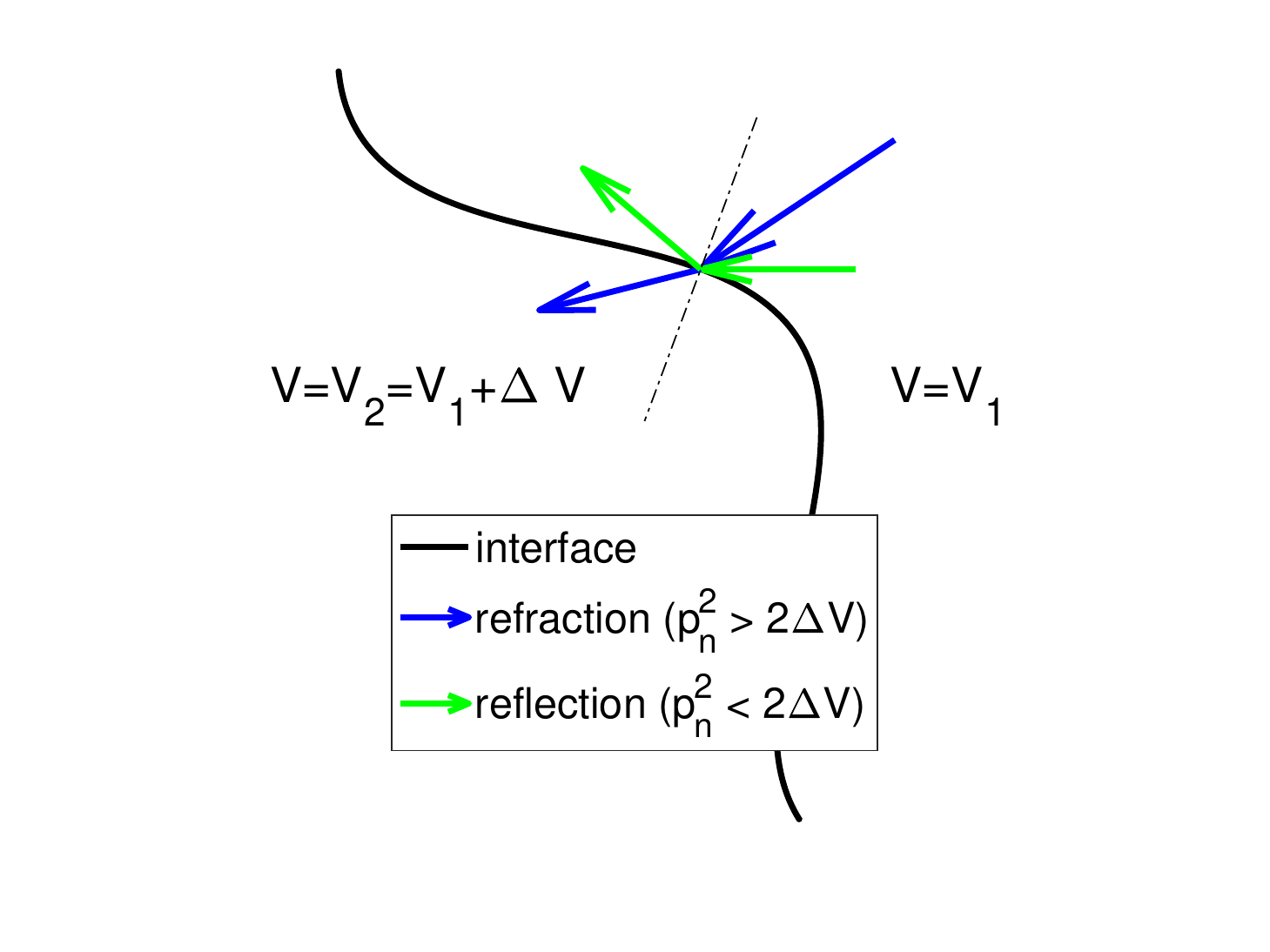}
\vspace{-20pt}
\caption{\small Two ways in which an \emph{impact} can change the momentum, 2D illustration.}
\label{fig_impactIllustration}
\end{figure}

After \emph{impact}, which is instantaneous in time (at $t_{k+1}$), the solution will be continued by \emph{flow} again, and these two types of behaviors alternate.

\begin{Remark}
	Unlike \emph{flow} which produces continuous trajectories, \emph{impact} creates discontinuities in $p$ ($q$ is still continuous).
\end{Remark}

\begin{Remark}
	It is not very meaningful to state if \emph{impact} conserves the total energy, because it is an instantaneous change of momentum, and at the impact, the position $q(t_{k+1})$ is on an interface where $V(\cdot)$ is undefined. However, if one considers the composition of an infinitesimal-time pre-impact flow, the impact, and another infinitesimal-time post-impact flow, then the composed map conserves the total energy $H$.
\end{Remark}

Important to mention is, the above definition of the exact solution requires two conditions, namely:

\begin{Condition}
	In the interested time horizon, the interface hitting position of the solution, $q(t_{k+1})$, only belongs to one interface $B_{i_{t_k}i_{t_{k+1}}}$ for each $k$. For example, the rare situation of an \emph{impact} at the intersection of three pieces illustrated in Fig. \ref{fig_impactDisallowed1} is assumed to never happen, because in this case there is no unique way of defining the post-impact momentum.
	\label{asmpt_oneInterface}
\end{Condition}

\begin{figure}[h!]
\center
\includegraphics[width=0.5\textwidth]{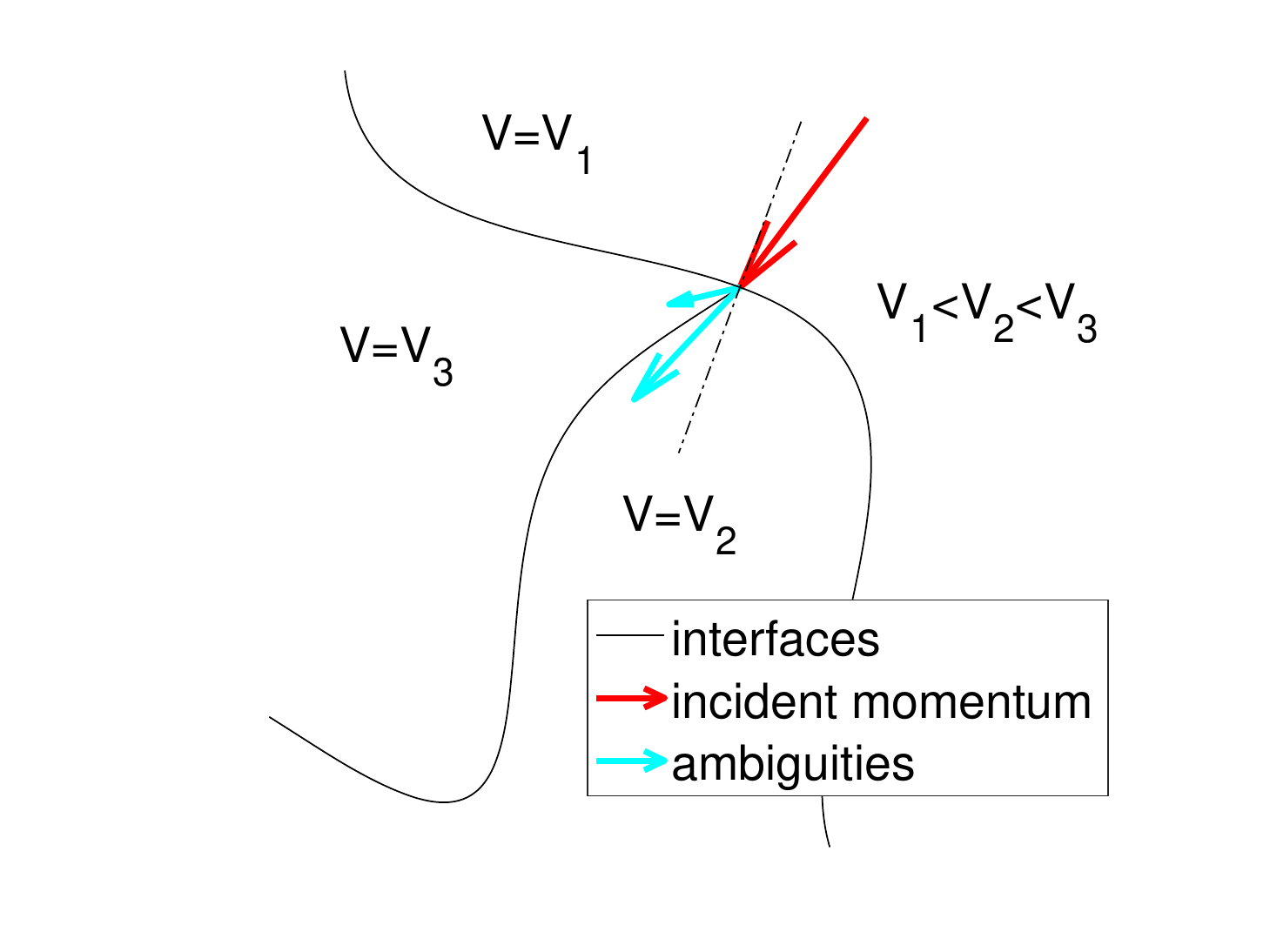}
\vspace{-20pt}
\caption{\small Ambiguity in post-impact momentum due to multiple-interface intersection.}
\label{fig_impactDisallowed1}
\end{figure}

\begin{Condition}
	For any $t$ in the interested time horizon such that the aforedefined $q(t)$ belongs to some interface $B_{ij}$, we have $p(t) \notin T_{q(t)}B_{ij}$. That is, sliding along an interface for a nonzero amount of time is assumed to never happen.
	\label{asmpt_noInterfaceSliding}
\end{Condition}

\begin{Remark}
	Because the discontinuity in our problem is in the scalar-valued potential $V(\cdot)$ instead in the vector field, we do not face challenges such as sliding motion in Filippov systems (e.g., \cite{filippov2013differential,leine2013dynamics,dieci2009sliding}), and defining a solution is easier.
\end{Remark}

If Cond.\ref{asmpt_noInterfaceSliding} is not satisfied, i.e., sliding along an interface occurs, one needs to use Geometrical Theory of Diffraction; see \cite{Jin-Yin}.

It is unclear, however, how to relax Condition \ref{asmpt_oneInterface} in a deterministic way. We feel that how to define a unique deterministic solution when Condition \ref{asmpt_oneInterface} fails will be problem dependent and requiring additional information about how the problem is set up. On the other hand, it is possible to define a stochastic solution by mimicking quantum mechanics; this is beyond the scope of the current work.

\subsubsection{Analytical solution for the quadratic case}
\label{sec_quadraticPotentialSln}
When the local Hamiltonian $\hat{H}=\frac{1}{2}p^T p+U(q)$ is integrable, the exact flow of the full problem $H=\hat{H}+V(q)$ is obtainable. As an example, this subsection will consider the case where $U$ is quadratic, and its exact solution will be used later for two purposes: (i) as part of a numerical algorithm (Section \ref{sec_3rdOrderSympl}), and (ii) as a benchmark for assessing numerical accuracy (Section \ref{sec_numericsBenchmark}).

For simplicity, consider one degree-of-freedom problems. Assume without loss of generality that $V$ corresponds to only 1 interface, i.e.,
\begin{equation}
	U(q)=\omega^2 (q-q_\text{off})^2/2,	\qquad 
	V(q) = \begin{cases}
			\Delta V, \qquad & q>q_\text{jump} \\
			0, \qquad & q<q_\text{jump} \\
			\text{undefined}, \qquad & q=q_\text{jump}
		   \end{cases}
	\label{eq_quadraticSolvableSetup}
\end{equation}
where $\omega, q_\text{off}, q_\text{jump}, \Delta V$ are constant scalar parameters.

Let $q,p$ denote the current position and momentum, and let $Q,P$ denote those after time $h$. Assume $h$ is small enough such that the interface is encountered at most once in time $h$ (if $h$ is large, break it into smaller time lapses and iterate the following). $Q,P$ can be obtained in the following way:

\noindent\hrulefill
\begin{algorithm}[H]
   compute a position proposal: $\tilde{q}=q_\text{off}+\cos(\omega h)(q-q_\text{off})+\sin(\omega h) p/\omega$, which corresponds to the new position when no interface crossing happens \;
   \eIf {$(q-q_\text{jump})(\tilde{q}-q_\text{jump})<0$, \emph{i.e., interface is crossed, }} {
   	compute $\hat{p}$ the pre-\emph{impact} momentum: $\hat{p}=\sigma \sqrt{\omega^2 (q-q_\text{off})^2+p^2-\omega^2 (q_\text{jump}-q_\text{off})^2}$, where $\sigma=-1$ if $q-q_{jump}>0$ and $\sigma=1$ if $q-q_{jump}<0$ \;
   	compute $t$ the time to \emph{impact}: let $t_1=\text{atan2}(\hat{p}/\omega, q_\text{jump}-q_\text{off}) - \text{atan2}(p/\omega, q-q_\text{off})$, $t_2=\begin{cases} t_1 &\text{if }t_1\geq 0 \\ t_1+2\pi &\text{if } t_1<0 \end{cases}$, and $t=(2\pi-t_2)/\omega$ \;
   	use post-\emph{impact} position $\bar{q}=q_\text{jump}$ and compute post-\emph{impact} momentum $\bar{p}$: \\
   	\eIf {$q-q_\text{jump}<0$, \emph{i.e. crossing is from left to right, }} { 
   		\eIf {$\omega^2 (q-q_\text{off})^2 + p^2 > 2\Delta V + \omega^2 (q_\text{jump}-q_\text{off})^2$ } {
   			$\bar{p}=\sqrt{\omega^2 (q-q_\text{off})^2 + p^2 - 2\Delta V - \omega^2 (q_\text{jump}-q_\text{off})^2}$, i.e., \emph{refraction} \;
   		} {
   			$\bar{p}=-\sqrt{\omega^2 (q-q_\text{off})^2 + p^2 - \omega^2 (q_\text{jump}-q_\text{off})^2}$, i.e., \emph{reflection} \;
   		}
   	}{	
   		\eIf {$\omega^2 (q-q_\text{off})^2 + p^2 > -2\Delta V + \omega^2 (q_\text{jump}-q_\text{off})^2$ } {
   			$\bar{p}=-\sqrt{\omega^2 (q-q_\text{off})^2 + p^2 + 2\Delta V - \omega^2 (q_\text{jump}-q_\text{off})^2}$, i.e., \emph{refraction} \;
   		} {
   			$\bar{p}=\sqrt{\omega^2 (q-q_\text{off})^2 + p^2 - \omega^2 (q_\text{jump}-q_\text{off})^2}$, i.e., \emph{reflection} \;
   		}
   	}
   	$Q = q_\text{off}+\cos(\omega (h-t))(\bar{q}-q_\text{off})+\sin(\omega (h-t))\bar{p}/\omega$, \\
   	$P=-\omega\sin(\omega (h-t))(\bar{q}-q_\text{off})+\cos(\omega (h-t)) \bar{p}$, \quad i.e., \emph{flow} $t$-time after \emph{impact}.
   }{
   	$Q = q_\text{off}+\cos(\omega h)(q-q_\text{off})+\sin(\omega h) p/\omega$, \\
   	$P = -\omega\sin(\omega h)(q-q_\text{off})+\cos(\omega h) p$, \quad i.e., no \emph{impact}, \emph{flow} $h$-time only\;
   }
\end{algorithm}
\noindent\hrulefill

\section{The numerical methods}
\subsection{A first-order in position time-reversible symplectic integrator}
\label{sec_1stOrderSympl}
A symplectic integrator for \eqref{eq_canonicalHamiltonian} can be constructed via the approach of Hamiltonian splitting and composition \cite{Hairer06}.

More precisely, denote by $\phi_1^\delta$ the $\delta$-time flow of the Hamiltonian $H_1=U(q)$, and by $\phi_2^\delta$ the $\delta$-time flow of $H_2=\frac{1}{2}p^Tp + V(q)$. Although the exact flow of $H$, $\phi^\delta$, is generally not numerically obtainable, the actions of $\phi_1^\delta$ and $\phi_2^\delta$ can be exactly obtained in explicit forms. Then appropriate compositions of $\phi_1$ and $\phi_2$ (e.g., Integrator \ref{int_1stOrderSymplectic}) will provide symplectic approximations of $\phi$ with vanishing error as $\delta\rightarrow 0$.

\smallskip

More specifically, $\phi_1$ is easy to evaluate:
\begin{equation}
	\phi_1^\delta: [q,p] \mapsto [q,p-\delta \nabla U(q)]
	\label{eq_phi1}
\end{equation}
On the other hand, $\phi_2^\delta [q,p]$ can be obtained by evolving the exact flow of $H_2$ for $\delta$-time. Section \ref{sec_exactSln} described how to do so by alternating \emph{flow} and \emph{impact} phases. In fact, $\phi_2$ is analytically computable, because $H_2$ does not contain the nonlinear potential $U(\cdot)$, and each \emph{flow} phase simply corresponds to free drift in a straight line.

Therefore, the only nontrivial parts for evaluating $\phi_2$ are (i) to compute, given $q,p$ vectors, how much time a particle at position $q$ with momentum (same as velocity) $p$ first hits one of the known interfaces, and (ii) to alter the momentum afterwards. How to compute the first hitting time depends on how the interfaces is provided in the problem setup. 

\paragraph{The demonstrative case in which the interface geometry is analytically known.} In this case, we can always find an affine transformation of $q$ and an associated linear transformation of $p$ (for making the transformation of both $q$ and $p$ canonical), such that $p$ is rotated to align with the $x$-axis, $q$ is on the $x$-axis, and the relevant interface\footnote{Recall: when $\delta$ is small enough, there will be at most one interface encountered.} passes through the origin; denote the unit tangent vector of the interface at origin by $\hat{t}$.

Then $[Q,P]=\phi_2^\delta[q,p]$ will be given by the following steps: first, let 
\begin{equation}
	\tau=-q/p,
	\label{eq_phi2hittingTime}
\end{equation}
where the division is understood as a ratio between their $x$-components; $\tau$ is the time to hit the interface. If $\tau>\delta$ or $\tau<0$, i.e., no \emph{impact} within this step, then let
\begin{equation}
	Q=q+\delta p, \qquad P=p
	\label{eq_phi2noImpact}
\end{equation}
and $\phi_2$ computation is completed; otherwise, let
\begin{equation}
	q_a=0,	\qquad p_a=p
	\label{eq_phi2preImpact}
\end{equation}
be the result of the pre-impact \emph{flow}, let
\[
	p_{at}=(p_a\cdot \hat{t})\hat{t}, \quad \text{and}\quad p_{an}=p_a-p_{at}
\]
be the tangential and normal components of the incident momentum at the interface, denote by $\Delta V=(V(0^-)-V(0^+))\text{sgn}(\hat{x}\cdot q)$ the potential jump across the interface, let
\begin{equation}
	q_b=0, \qquad p_b=\begin{cases} p_{at}-p_{an}	& \qquad \|p_{an}\|^2/2 < \Delta V \\
									p_{at}+\frac{p_{an}}{\|p_{an}\|} \sqrt{\|p_{an}\|^2-2\Delta V} &\qquad \|p_{an}\|^2/2 > \Delta V
					  \end{cases}
	\label{eq_phi2impact}
\end{equation}
be the result of \emph{impact} (the first case is \emph{reflection} and the second \emph{refraction}), and let
\begin{equation}
	Q=(\delta-\tau)p_b,	\qquad P=p_b
	\label{eq_phi2postImpact}
\end{equation}
be the result of the post-impact \emph{flow}. The result of $\phi_2^\delta$ will be $Q,P$.

It is not difficult to see that the same calculation works in general coordinate systems as long as the interface geometries are simple enough such that the first hitting time $\tau$ is a computable function of $q$ and $p$.

When the interface geometry is too complex to be explicitly and analytically characterized, $\tau$ can be numerically computed to machine precision rapidly. See Section \ref{sec_timeToImpact} for details when interfaces are provided by either (i) level sets of $\mathcal{C}^1$ functions, or (ii) discontinuous $V(\cdot)$ values.

\paragraph{Two additional comments.}
Necessary to further clarify is, in order to have guaranteed accuracy of our numerical methods, $\delta$ needs to be sufficiently small, not only for controlling the error of the composition, but also for ensuring there is at most one \emph{impact} per $\delta$-sized time step. We will also describe a numerical method that can account for multiple \emph{impacts} per step; however, whether all \emph{impacts} within a step can be detected depends on how interfaces are provided. If the numerical algorithms in Section \ref{sec_timeToImpact} need to be used, there is no guarantee on the capture of the earliest \emph{impact} if $\delta$ is too large.

It is also notable that although the above description of $\phi_2^\delta$ evaluation might appear `discrete', it is \textbf{exact}. In practice, of course, its accuracy will be limited by the machine precision. It is possible to algorithmically alleviate some limitation of machine precision and we refer to the innovative idea in \cite{higham1993accuracy}, but in this article we will equate `accurate to machine precision' with `exact'.

\bigskip

\noindent
As the explicit evaluations of $\phi_1$ and $\phi_2$ are obtained, a numerical integrator can be constructed by composing these maps. The method proposed in this section is a one-step method that uses a constant step size of $\delta$ and a one-step update given by
\begin{Integrator}[1st-order in position, time-reversible, symplectic]
Use one step update
\[
	[q_{n+1},p_{n+1}] = \phi_1^{\delta/2} \circ \phi_2^\delta \circ \phi_1^{\delta/2} [q_n,p_n],
\]
where $\phi_1$ is given by \eqref{eq_phi1} and $\phi_2$ is given by (\ref{eq_phi2hittingTime}--\ref{eq_phi2postImpact}).
	\label{int_1stOrderSymplectic}
\end{Integrator}

\paragraph{Symplecticity.}
Obviously $\phi_1^\delta$ defined in \eqref{eq_phi1} is a symplectic map. If $\phi_2^\delta$ is also symplectic as it should be (since it is the exact flow of some Hamiltonian system), then Integrator \ref{int_1stOrderSymplectic} is symplectic because it is the composition of symplectic maps. And indeed $\phi_2^\delta$ is symplectic. In fact, both the `no impact' case \eqref{eq_phi2noImpact} and the \emph{reflection}/\emph{refraction} case correspond to symplectic maps. The former is obvious, and for the latter we have:

\begin{Theorem} Under Conditions \ref{asmpt_oneInterface} and \ref{asmpt_noInterfaceSliding}, the \emph{reflection}/\emph{refraction} case corresponds to symplectic $\phi_2^\delta$.
\end{Theorem}
\begin{proof}
Substituting \eqref{eq_phi2hittingTime}, \eqref{eq_phi2preImpact}, \eqref{eq_phi2impact} into \eqref{eq_phi2postImpact} and computing the Jacobian $J=d[Q,P]/d[q,p]$, one can verify that $J^T \Omega J=\Omega$ where $\Omega=\begin{bmatrix} 0 & I \\ -I & 0 \end{bmatrix}$ for each case of \eqref{eq_phi2impact}. This shows the preservation of the canonical symplectic 2-form in vector space.
\end{proof}

\begin{Remark}
    One may worry that when transforming an infinitesimal phase space volume by $\phi_2^\delta$, part of it undergoes \emph{reflection} and another part undergoes \emph{refraction}, which would challenge the above case-by-case demonstration of symplecticity. However, this possibility is ruled out by Condition \ref{asmpt_noInterfaceSliding}, because the transition between \emph{reflection} and \emph{refraction} corresponds to sliding along the interface.
\end{Remark}

Worth commenting, however, is that neither of the three submaps of $\phi_2$, namely pre-impact \emph{flow} \eqref{eq_phi2preImpact}, \emph{impact} \eqref{eq_phi2impact}, or post-impact \emph{flow} \eqref{eq_phi2postImpact}, is symplectic (simple algebra will show $J^T\Omega J \neq \Omega$). The intuition is, \eqref{eq_phi2preImpact} and \eqref{eq_phi2postImpact} are drifts, but unlike simple drifts over constant time which are symplectic, they have additional $q,p$ dependence through $\tau$ (see \eqref{eq_phi2hittingTime}), which breaks the symplecticity; for \eqref{eq_phi2impact}, both the reflection and refraction cases rescale one component of the momentum without changing anything else, and thus cannot be symplectic. Therefore, it is a nontrivial fact that {\it their composition, $\phi_2^\delta$, resumes to be symplectic}.

What will the symplecticity of the proposed integrator imply about its accuracy in numerical simulations? We do not yet have a theory. Traditionally, the favorable long time performances of symplectic integrators are supported by elegant theoretical guarantees such as backward error analysis (e.g., \cite{moser1968lectures, hairer1994backward, benettin1994hamiltonian, Hairer06}), linear error growth for integrable systems (e.g., \cite{calvo1993development,calvo1995accurate,quispel1998volume,Hairer06}), and near-preservation of adiabatic invariants \cite{hairer2000long, MR2275175, Hairer06}. However, none of them can be directly applied to the discontinuous Hamiltonian problem, mainly due to a lack of differentiability in both the Hamiltonian and the $p$ trajectory.
However, symplectic integrators proposed in this article were still observed to exhibit pleasant long time accuracy in numerical experiments (see Section \ref{sec_numerics}), and this is intuitive because symplecticity would still be desired at least in \emph{flow} phases, in which the numerical solution only fluctuates the energy instead of drifting it as many nonsymplectic methods may do, and for most of the time the particle is in \emph{flow} phases.

We imagine that a proof of long time accuracy might require a discontinuous generalization of canonical perturbation theory, which is beyond the scope of this article.

\paragraph{Order of accuracy, time-reversibility, and higher-order splitting schemes} Since $H=H_1+H_2$, Integrator \ref{int_1stOrderSymplectic} can be recognized as a Strang splitting method \cite{Strang:68}. In the traditional (smooth) theory this would suggest that $\phi^\delta=\phi_1^{\delta/2} \circ \phi_2^\delta \circ \phi_1^{\delta/2} + \mathcal{O}(\delta^3)$, i.e., Integrator \ref{int_1stOrderSymplectic} is 2nd-order due to a 3rd-order truncation error. However, that is \textbf{no longer the case} due to discontinuities in $p$ (momentum), and the method actually only has 2nd-order local truncation error in position (see Appendix \ref{appd_order1stOrderSymplectic} for a detailed demonstration).

More generally, a common technique for turning a 1st-order method into 2nd-order is to compose it with its adjoint (see e.g., Chap II.4 in \cite{Hairer06}). It is true that $\phi_1$ and $\phi_2$ are self-adjoint (a.k.a. symmetric or time-reversible) and they respectively form semigroups, and Integrator \ref{int_1stOrderSymplectic} thus may be seen as a 1st-order method $\psi^{\delta/2} := \phi_1^{\delta/2} \circ \phi_2^{\delta/2}$ composed with its adjoint $\left(\psi^{\delta/2}\right)^*= \phi_2^{\delta/2} \circ \phi_1^{\delta/2}$. However, the tradition proof of the increased order of $\psi^{\delta/2} \circ \left(\psi^{\delta/2}\right)^*$ relies on Taylor expansion, which is no longer applicable due to the discontinuity produced by $\phi_2$. Because of this, Integrator \ref{int_1stOrderSymplectic} loses some order of convergence, although the symmetric nature of its composition makes the method time-reversible.

One may also wonder if higher-order splitting schemes will work as designed, such as triple jump (whose one step update is given by $\phi_1^{\delta \gamma/2} \circ \phi_2^{\delta \gamma} \circ \phi_1^{\delta (1-\gamma)/2} \circ \phi_2^{\delta (1-2\gamma)} \circ \phi_1^{\delta (1-\gamma)/2} \circ \phi_2^{\delta \gamma} \circ \phi_1^{\delta \gamma/2}$, with $\gamma=1/(2-2^{1/3})$ \cite{creutz1989higher, forest1989canonical, suzuki1990fractal, Yoshida:90}, which normally is 4th-order, or more versions given in, for instance, \cite{McQu02, castella2009splitting, hansen2009high, blanes2013optimized, blanes2013new}. Unfortunately, they cannot produce anything beyond 1st-order (in position), again due to the lost of smoothness; the constant in the error bound may be improved though. Detailed proof by a counter-example is analogous to that in Appendix \ref{appd_order1stOrderSymplectic} and omitted.

Also important to note is, Integrator \ref{int_1stOrderSymplectic} only has a 1st-order local truncation error in momentum if the step includes a discontinuous momentum change. Without considering the problem's specific structure, this would imply a 0th-order global error, not only in momentum but also in the position variable as they are coupled. Such a `bad accuracy' is actually expected to some extent, because the momentum exhibits a jump discontinuity across each interface. To better explain this, suppose the actual time of an interface crossing is different from the crossing time of a numerical simulation, then no matter how close these two time points are and how accurate the numerical momentum is outside the interval limited by these two times, inside this time interval the numerical error in momentum is $\mathcal{O}(1)$, because there either the numerical solution or the exact solution has already completed an $\mathcal{O}(1)$ jump in momentum, but not both.

However, for our specific problem and specific integrator, `bad accuracy' is actually localized. Observed in all numerical experiments (Section \ref{sec_numerics}) was, the method still exhibited 1st-order global error in position, and while the momentum did not have 1st-order global error uniformly in time, away from interface crossing events the momentum was still 1st-order accurate. The intuition is, the numerical momentum will still catch up with the exact momentum after both solutions complete the interface crossing, and the effect of this lag will not be much amplified because we assumed there is no immediate consecutive interface crossings, and \emph{impact} and \emph{flow} have to alternate.

\subsection{Computing the time to \emph{impact} in complex situations}
\label{sec_timeToImpact}
Given $q,p$ and a generic explicit flow map $\Phi^h[q,p]$, which is either provided by an exact solution (the case of \eqref{eq_phi2noImpact}, used for example in the 1st-order method in Section \ref{sec_1stOrderSympl}), or provided by a numerical integrator (the case for the high-order construction in Section \ref{sec_highOrderNonSympl}), we can accurately and efficiently compute the time $\tau$ such that the position component of $\Phi^\tau[q,p]$ exactly hits an interested interface $B_{ij}$. In \cite{JinWuHuang} a linear interpolation was used to approximate $\tau$. Here we will compute it more accurately. 
Two cases will be discussed.

For simplified notation, $Q(h)$ will denote the position component of $\Phi^h[q,p]$.
\begin{itemize}
\item
\textbf{When the interface is given by a level set.} If $B_{ij}=f^{-1}(\{0\})$ for some known $\mathcal{C}^1$ function $f(\cdot)$, and $Q(h)$ is a $\mathcal{C}^1$ function (any exact flow of a continuous system satisfies this property, and so should any reasonable approximation of an exact flow), we use \textbf{Newton's method} to solve for $\tau$ in
\[
	f(Q(\tau))=0.
\]
The solution is given by the simple iteration
\[
	\tau_{k+1}=\tau_k - \frac{f(Q(\tau_k))}{f'(Q(\tau_k))Q'(\tau_k)}.
\]
Under standard (reasonable) assumptions this iteration converges quadratically. This means that $\tau$ can be computed to machine precision by a small amount of iterations. Since $\tau$ is just 1-dimensional, these iterations are computationally cheap. Initialization can be made efficient too, for instance via
\[
	\tau_0=-\delta f(Q(0))/f(Q(\delta)),
\]
derived from $0=f(Q(\tau))\approx f(Q(0))+\frac{\tau}{\delta} f(Q(\delta))$.

This case is demonstrated by the numerical experiment in Section \ref{sec_numericsModifiedKepler}.

\item
\textbf{When the interface is only known indirectly through $V(\cdot)$ values,} we use \textbf{the bisection method}: Initially, let $t_l=0$ and $t_r=\delta$. At each iteration, let $t_m=(t_l+t_r)/2$ and compare $V(Q(t_m))$ with $V(Q(t_l))$ and $V(Q(t_r))$; if it agrees with the left value, update by $t_r \leftarrow t_m$, and other update by $t_k \leftarrow t_m$. Terminate the iteration when $t_r - t_l < \epsilon$ for some preset small $\epsilon$, and let $\tau=t_m$.

This way, no derivative information is needed, and the convergence of estimated $\tau$ values is linear, which implies an exponential decay of error. This speed of convergence is slower than that of Newton, but difficult to improve in this setup because one can only evaluate $V$'s values at chosen locations and $V$ is piecewise constant. On the other hand, to obtain any prefixed accuracy, one still only needs logarithmically many iterations. Therefore, machine precision can again be achieved with a small computational budget.

This case is demonstrated by the numerical experiment in Section \ref{sec_numericsModifiedKepler}.

\end{itemize}

\subsection{A third-order in position time-reversible symplectic integrator, when there is only one interface and one degree of freedom}
\label{sec_3rdOrderSympl}
When $U$ is at least a $\mathcal{C}^3$ function, position is one-dimensional and there is only one interface, we are able to increase the order of convergence while maintaining the symplecticity. To do so, denote the interface location by $q_\text{jump}$, and we first decompose the smooth nonlinear potential $U$ into a quadratic approximation centered at $q_\text{jump}$ and a nonlinear correction: let
\begin{align*}
	U_\text{quad}(q) &:= U(q_\text{jump}) + \frac{dU}{dq}(q_\text{jump})\big(q-q_\text{jump}\big) + \frac{1}{2}\frac{d^2 U}{dq^2}(q_\text{jump}) \big(q-q_\text{jump}\big)^2 , \\
	U_\text{corr}(q) &:= U(q)-U_\text{quad}(q) .
\end{align*}
Then we split the original Hamiltonian as the sum of a discontinuous quadratic problem and a correction: let 
\begin{equation}
	H_1=\frac{1}{2}p^T p+V(q)+U_\text{quad}(q)		\quad \text{and} \quad
	H_2=U_\text{corr}(q)
	\label{eq_SplitHamiltoniansFor3rdOrder}
\end{equation}
and denote by $\phi_1^\delta$ and $\phi_2^\delta$ respectively their exact solution flows. Then both are symplectic maps and explicitly obtainable: the latter is a simple drift in the momentum, and the former is given in Section \ref{sec_quadraticPotentialSln}. More precisely, to convert to notations in Section \ref{sec_quadraticPotentialSln}, it is easy to see
\begin{align*}
& U_\text{quad} = \frac{1}{2} \omega^2 (q-q_\text{off})^2 + (an~unimportant~constant), \\
\text{where} \quad & \omega^2=U''(q_\text{jump}),
\qquad
q_\text{off}=q_\text{jump}-U'(q_\text{jump})/U''(q_\text{jump}).
\end{align*}

Finally, we combine this splitting and the classical idea of triple-jump to obtain a method whose one step update is given by:

\begin{Integrator}[3rd-order in position, time-reversible, symplectic]
Use one step update
\[
	[q_{n+1},p_{n+1}] = \phi_2^{\delta \gamma/2} \circ \phi_1^{\delta \gamma} \circ \phi_2^{\delta (1-\gamma)/2} \circ \phi_1^{\delta (1-2\gamma)} \circ \phi_2^{\delta (1-\gamma)/2} \circ \phi_1^{\delta \gamma} \circ \phi_2^{\delta \gamma/2} [q_n,p_n],
\]
where $\gamma=1/(2-2^{1/3})$.
	\label{int_3rdOrderSymplectic}
\end{Integrator}

\begin{Remark}[order reduction]
	This method does not have a 4th-order global error as a continuous analogue would have, but numerically observed is that it has 3rd-order global accuracy in terms of $q$, and the global error of $p$ is also 3rd-order whenever at a time point sufficiently away from \emph{impact}.
\end{Remark}

\begin{Remark}[an alternative approach]
	Triple-jump is a classical way of obtaining a 4th-order smooth integrator, but it is not the only approach. Another classical example is Suzuki's fractal \cite{suzuki1990fractal}, which in our case is
\[
	\phi_2^{\delta \gamma/2} \circ \phi_1^{\delta \gamma} \circ \phi_2^{\delta \gamma} \circ 
	\phi_1^{\delta \gamma} \circ
	\phi_2^{\delta (1-3\gamma)/2} \circ \phi_1^{\delta (1-4\gamma)} \circ \phi_2^{\delta (1-3\gamma)/2} \circ 
	\phi_1^{\delta \gamma} \circ
	\phi_2^{\delta \gamma} \circ \phi_1^{\delta \gamma} \circ \phi_2^{\delta \gamma)/2}
\]
where $\gamma=1/(4-4^{1/3})$. It often leads to error at the same order but with smaller prefactor, however at the expense of using more stages.

Suzuki's fractal doesn't directly transfer to our non-smooth setup either. If one uses the $\phi_1$ and $\phi_2$ constructed in Sec.\ref{sec_1stOrderSympl} for generic problems, the resulting method remains only 1st-order. With the specialized $\phi_1$ and $\phi_2$ in this section, the result will be 3rd-order (in position, away from \emph{impact}), same as triple jump (note the order reduction).

Symplecticity and reversibility, however, can be obtained.
\end{Remark}

\begin{Remark}[2nd-order in position, time-reversible, symplectic integrator]
	If a one-step update of $\phi_2^{\delta/2} \circ \phi_1^\delta \circ \phi_2^{\delta/2}$ is used instead (i.e., Strang splitting), numerical experiments suggest that the method has a 2nd-order global error; that is, no order was lost. This is an improvement of the generic method in Section \ref{sec_1stOrderSympl}.
	\label{int_2ndOrderSymplectic}
\end{Remark}

Our intuition behind the reduced order (4$\to$3) is, away from interfaces the triple-jump and the Strang splitting will have regular 5th-order and 3rd-order local truncation errors, and their truncation errors degrade to lower order only when the current step involves an \emph{impact}. However, to encounter an \emph{impact}, the current position must be $\mathcal{O}(h)$ away from the interface, which means, according to Taylor expansion, $U_\text{corr}$ is $\mathcal{O}(h^3)$ and $U'_\text{corr}$ is $\mathcal{O}(h^2)$. This exact scaling allows the splitting (global) accuracy to increase from 1st-order (see Section \ref{sec_1stOrderSympl}) to 3rd-order near the \emph{impact}. If a splitting scheme of order higher than 3 in a smooth setup is used (e.g., triple jump), 3rd-order will be retained; in the case of Strang splitting, itself is only 2nd-order, and therefore its discontinuous version based on \eqref{eq_SplitHamiltoniansFor3rdOrder} is still just 2nd-order due to limitations of non-\emph{impact} steps.

The aforementioned methods based on triple jump and Strang splitting, or any symmetric composition of $\phi_1$ and $\phi_2$, are reversible.

\begin{Remark}
The method in this section generalizes to only a small subset of multi-degree-of-freedom problems: if there is only one interface and it is linear, then we can similarly construct $U_{quad}$ by expanding $U$ in the direction normal to the interface; if there are multiple interfaces and all of them are linear, a Voronoi decomposition may help construct a continuous, piecewise quadratic $U_{quad}$; however, it is unclear whether the idea in this subsection can work for nonlinear interface(s).
\end{Remark}

\subsection{High-order time-reversible but nonsymplectic integrators}
\label{sec_highOrderNonSympl}
Although it is difficult to obtain a high-order symplectic method due to discontinuity in momentum (see discussions in Sections \ref{sec_1stOrderSympl} and \ref{sec_3rdOrderSympl}), an event-driven approach can be adopted to construct arbitrarily high-order time-reversible methods. It is not clear yet what would be the disadvantage of losing symplecticity, as numerical experiments will also demonstrate pleasant long-time properties, similar to those given by the long time error analysis of smooth reversible integrator for reversible integrable systems \cite{Hairer06}; note however that a theoretical explanation is lacking.

The one-step update of the new method, with step size denoted by $\delta$, will be constructed based on an $l$-th order symplectic integrator for the smooth Hamiltonian $\hat{H}=\frac{1}{2}p^T p+U(q)$. The construction of such integrators is well known (see e.g., \cite{Hairer06}, or \cite{Ta16} for a recap). Denote the one-step update of this integrator by $\psi^\delta$, where $\delta$ is the step size. We will also assume the explicit existence of a function $I(q)$ that return $i$ if $q\in D_i$, the $i$-th component of the partition of position space into smooth regions (see its definition below \eqref{eq_canonicalHamiltonian}). Then a high-order integrator for the discontinuous problem $H$ can be constructed through the following stages:
\begin{Integrator}[$l$-th order in position, time-reversible (if the legacy method $\psi$ is reversible)] ~
\begin{enumerate}
\item \label{item_stage1}
	Given the current d-dimensional position $q$ and momentum $p$, \textbf{check if an \emph{impact} will occur in $\delta$ time.} To do so, compute $[Q,P]=\psi^\delta[q,p]$. 
	If $I(q)\neq I(\hat{q})$, at least one \emph{impact} occurred.
	
	If no impact occurred, update $[q,p]$ by $[Q,P]$, and go back to Stage \ref{item_stage1} for the next step. Otherwise, continue and denote by $B_{ij}$ the relevant interface $B_{I(q)I(\hat{q})}$.
	
\item
	Otherwise, \textbf{numerically approximate the time $\tau$ (under the flow of $\hat{H}$) to \emph{impact}}. Depending on how the interface $B$ is provided, this first hitting time $\tau$ may be explicitly computable, or it needs to be numerically estimated. For the latter case, Section \ref{sec_timeToImpact} provides details about $\tau$'s rapid computation to machine precision, when the interface is provided (i) as a level set of a $\mathcal{C}^1$ function, or (ii) in the general case, purely through values of the discontinuous potential $V(\cdot)$.

\item \label{item_stage3}
	Update $q$,$p$ using \textbf{one step of the continuous symplectic integrator with step size $\tau$}; i.e,
	\[
		[q,p] \leftarrow \psi^\tau[q,p]
	\]
	
\item \label{item_stage4}
	Compute the action of an \emph{impact}. That is, based on (i) the jump size of $V$ given by the interface $B_{ij}$, and (ii) the normal vector of $B_{ij}$ at position $q$ from $D_i$ to $D_j$, perform an instantaneous \textbf{update of $p$ either via \emph{refraction} (eq. (\ref{eq_refraction})) or \emph{reflection} (eq. (\ref{eq_reflection}))}.
	
	Note: if $B_{ij}$ is only implicitly described by $V$ values, the normal vector has to be numerically searched, but this search can again be done efficiently and to machine precision using the bisection method, which will take poly-log time.

\item \label{item_stage5}
	Update $q$,$p$ again using \textbf{one step of the continuous symplectic integrator, this time with step size $\delta-\tau$}; i.e,
	\[
		[q,p] \leftarrow \psi^{\delta-\tau}[q,p]
	\]
\end{enumerate}
\label{int_lstOrder}
\end{Integrator}

\begin{Remark}[Applicability]
	As long as $\delta$ is small enough, there is at most one \emph{impact} per step under Conditions \ref{asmpt_oneInterface}, \ref{asmpt_noInterfaceSliding} and the assumption that $D_i$'s are open sets.
\end{Remark}

\begin{Remark}[Non-symplecticity]
	(i) If $\tau$ were a constant (independent of $q,p$), then Stage \ref{item_stage3} and Stage \ref{item_stage5} are both symplectic updates; however, Stage \ref{item_stage4} is not symplectic.
	(ii) If $\tau$ were the exact hitting time ($q,p$ dependent) and $\psi$ were the exact flow of $\hat{H}$, the new method would be symplectic, but this is unlikely to be possible because then the new method is in fact exact. In this case, neither Stage \ref{item_stage3}, \ref{item_stage4} or \ref{item_stage5} is symplectic, but their composition is.
	(iii) For the new method, Stage \ref{item_stage4} remains non-symplectic, neither Stage \ref{item_stage3} or \ref{item_stage5} is symplectic because $\tau$ depends on $q,p$, and the composition of stages is generally not symplectic.
\end{Remark}

\begin{Remark}[Order of accuracy]
The accuracy of the new method is numerically observed in experiments to be $l$-th order (same as the legacy smooth integrator $\psi$), however in a modified sense that the global error of position is $\mathcal{O}(\delta^l)$, and that of momentum is $\mathcal{O}(\delta^l)$ at time points away from interface crossings (the accuracy in $p$ is 0th-order near interface crossings). The intuition is, during simulation till a fixed $\mathcal{O}(1)$ time, there are only finitely many interface crossing events, which means the $\mathcal{O}(\delta^l)$ error of $\psi$ will only be amplified at most by a constant factor.
\end{Remark}

\begin{Remark}[Reversibility]
The new method can be easily checked to be reversible if (i) $\psi$ is reversible, (ii) the hitting time is solved for exactly, (iii) there is at most one interface crossing per step.
\end{Remark}

\subsection{The adaptive version for safer usages of large timesteps}
\label{sec_highOrderNonSymplAdaptive}
Depending on the problem, there might be a regime of $\delta$ values that are small enough for resolving the dynamics generated by the smooth Hamiltonian $\hat{H}$, however too large in the sense that multiple interface crossings can be encountered within one step. This can happen, for instance, when an interface has a large curvature, which may result in consecutive \emph{impact}s that are close in time. In this case, methods proposed cannot be accurate when the step size $\delta$ is large.

To avoid this deterioration of accuracy, one could of course simply reduce $\delta$, so that at most one interface crossing occurs per step. However, this is not the most computationally efficient solution. Instead, we propose to employ an adaptive time-stepping approach. Our way of introducing adaptivity will destroy symplecticity, and therefore it will be demonstrated on the high-order nonsymplectic method in Section \ref{sec_highOrderNonSympl}. Most parts will be the same as before, except for a handful of modifications which are underlined:

\begin{Integrator}[$l$-th order in position, adaptive]
\uline{At the beginning of each step, let $\hat{\delta}=\delta$. Then,}
\begin{enumerate}
\item
	Given the current d-dimensional position $q$ and momentum $p$, \textbf{check if an \emph{impact} will occur in \uline{$\hat{\delta}$} time.} To do so, compute $[Q,P]=\psi^\delta[q,p]$. 
	If $I(q)\neq I(\hat{q})$, at least one \emph{impact} occurred.
	
	If no impact occurred, update $[q,p]$ by $[Q,P]$, and \uline{the current step is completed. Otherwise, continue.}

\item
	\textbf{Numerically approximate the time $\tau$ (under the flow of $\hat{H}$) to \uline{the first \emph{impact}}}. Depending on how the interface $B$ is provided, this first hitting time $\tau$ may be explicitly computable, or it needs to be numerically estimated. For the latter case, Section \ref{sec_timeToImpact} provides details about $\tau$'s rapid computation to machine precision, when the interface is provided (i) as a level set of a $\mathcal{C}^1$ function, or (ii) in the general case, purely through values of the discontinuous potential $V(\cdot)$. \uline{Note that there is no guarantee those numerical estimations really correspond to the first hitting time, if $\hat{\delta}$ is too large}.

\item
	Update $q$,$p$ using \textbf{one step of the continuous symplectic integrator with step size $\tau$}; i.e,
	\[
		[q,p] \leftarrow \psi^\tau[q,p]
	\]

\item
	Compute the action of an \emph{impact}. That is, based on (i) the jump size of $V$ given by the interface $B_{ij}$, and (ii) the normal vector of $B_{ij}$ at position $q$ from $D_i$ to $D_j$, perform an instantaneous \textbf{update of $p$ either via \emph{refraction} (eq. (\ref{eq_refraction})) or \emph{reflection} (eq. (\ref{eq_reflection}))}.
	
	Note: if $B_{ij}$ is only implicitly described by $V$ values, the normal vector has to be numerically searched, but this search can again be done efficiently and to machine precision using the bisection method, which will take poly-log time.

\item
	\uline{Update the remaining time using $\hat{\delta} \leftarrow \hat{\delta} - \tau$, and return to Stage 1.}
\end{enumerate}

\label{int_adaptive}
\end{Integrator}

\begin{Remark}
	This method no longer requires at most one \emph{impact} per step, and it is effective as long as the computed hitting time corresponds to the first hitting. It is not symplectic. Order of accuracy is numerically observed to be $l$ in the same sense as before. Reversibility is achieved if (i) $\psi$ is reversible, and (ii) every time the first (backward) hitting time is exactly solved for.
\end{Remark}

The example in Section \ref{sec_numericsMushroom} showcases the efficacy of this adaptive method. There the interface actually has kinks (corresponding to corners). Numerically these corners are never exactly reached, but \emph{impacts} near them can be clustered in time.

\section{Numerical experiments}
\label{sec_numerics}

\subsection{Quantification via a quadratic benchmark problem}
\label{sec_numericsBenchmark}


\paragraph{Setup.}
Consider a one degree-of-freedom problem where the smooth part of the potential is a quadratic function and the jump part corresponds to only one interface, i.e., \eqref{eq_quadraticSolvableSetup}. This is an analytically solvable case (see Sec.\ref{sec_quadraticPotentialSln}), and we chose it so that numerical and exact solutions can be compared to precisely quantify long time accuracy and numerical order of convergence.

Here we use parameters $\omega=2, q_\text{off}=1, \Delta V=3, q_\text{jump}=2$.

\begin{figure}
	\centering
    \begin{subfigure}[t]{0.49\textwidth}
		\includegraphics[width=\textwidth]{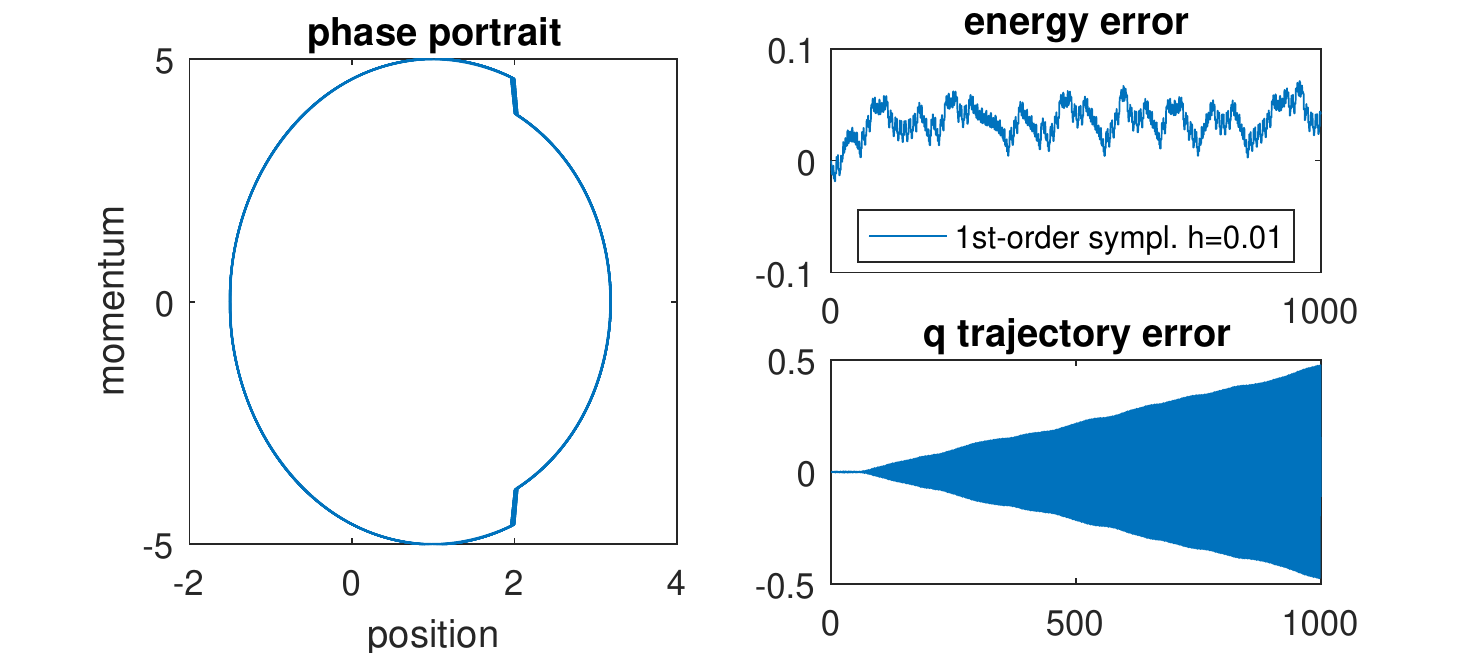}
	\end{subfigure}
    \begin{subfigure}[t]{0.49\textwidth}
		\includegraphics[width=\textwidth]{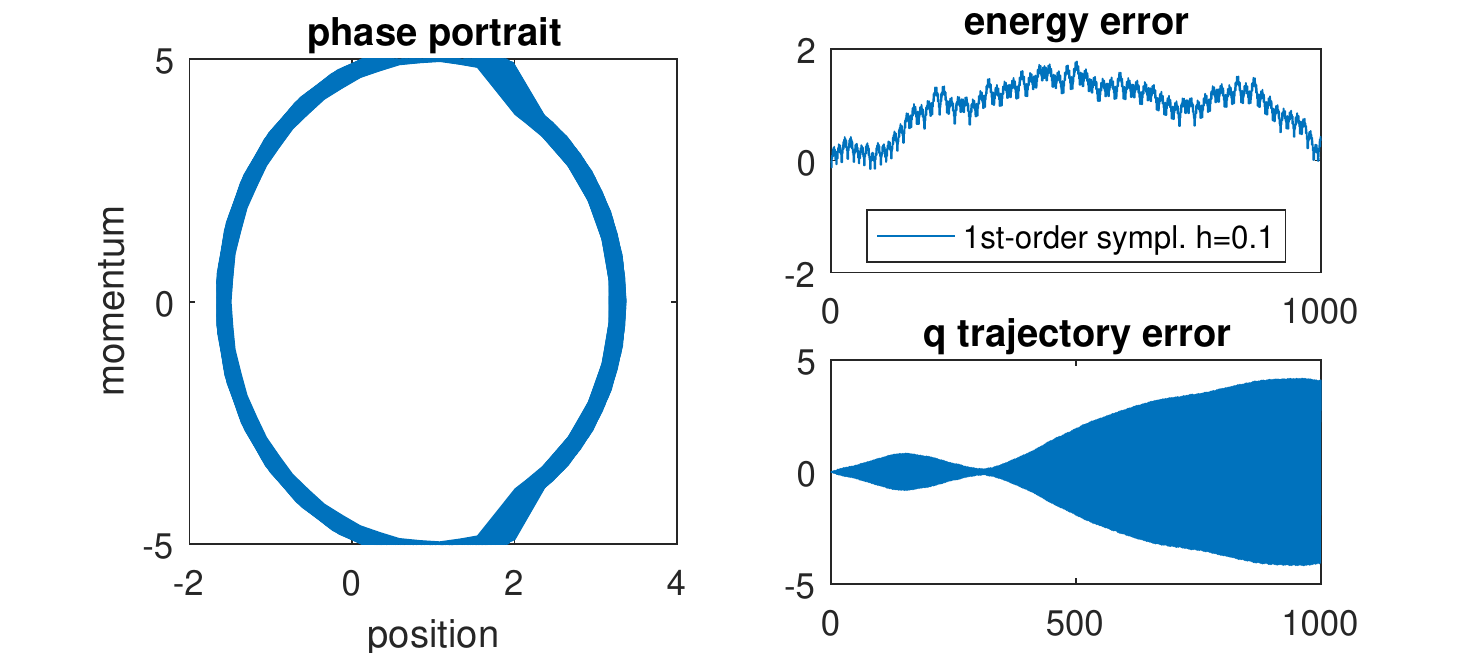}
	\end{subfigure}
    \begin{subfigure}[t]{0.49\textwidth}
		\includegraphics[width=\textwidth]{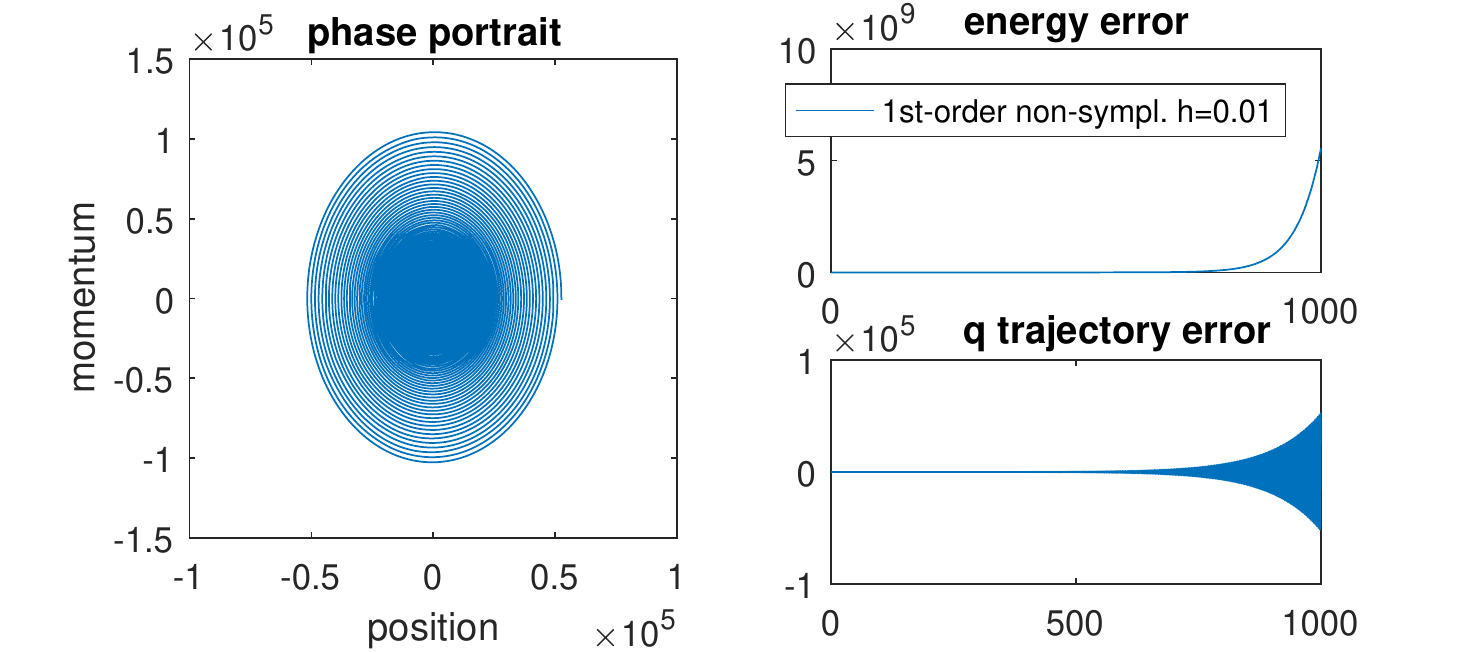}
	\end{subfigure}
    \begin{subfigure}[t]{0.49\textwidth}
		\includegraphics[width=\textwidth]{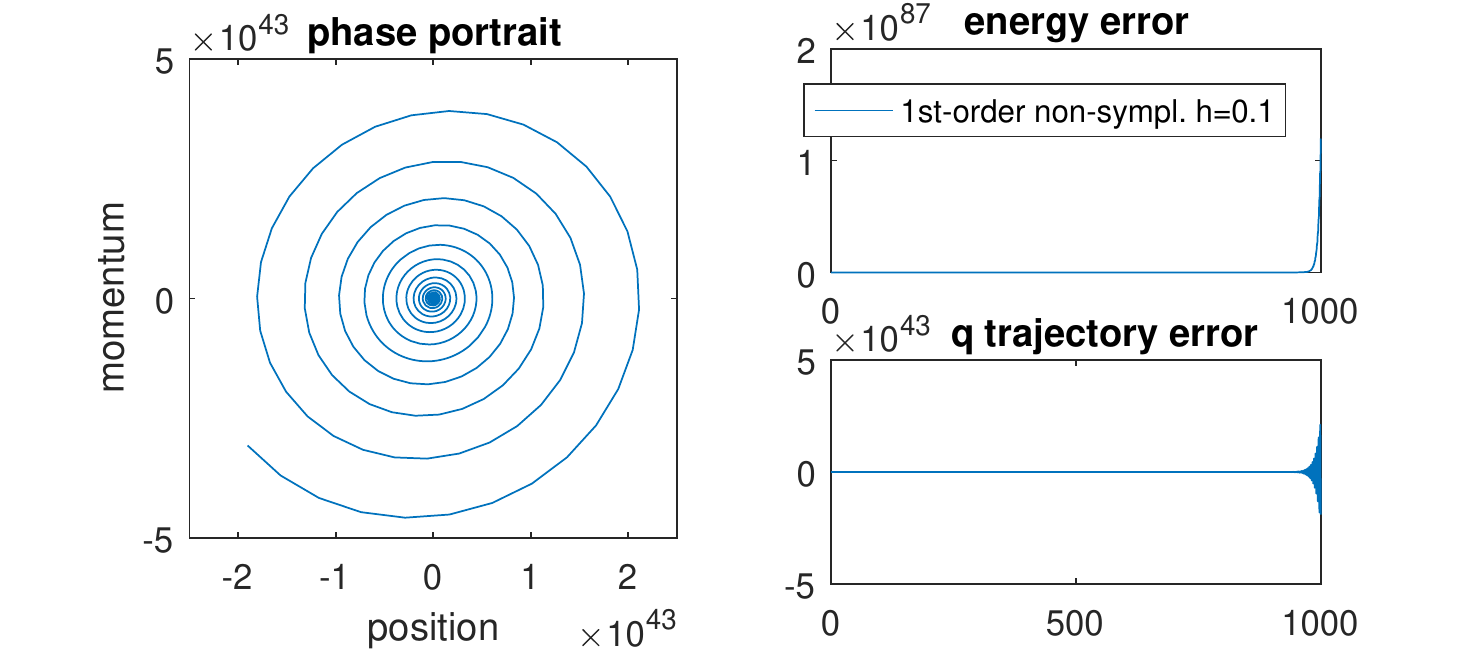}
	\end{subfigure}
    \begin{subfigure}[t]{0.49\textwidth}
		\includegraphics[width=\textwidth]{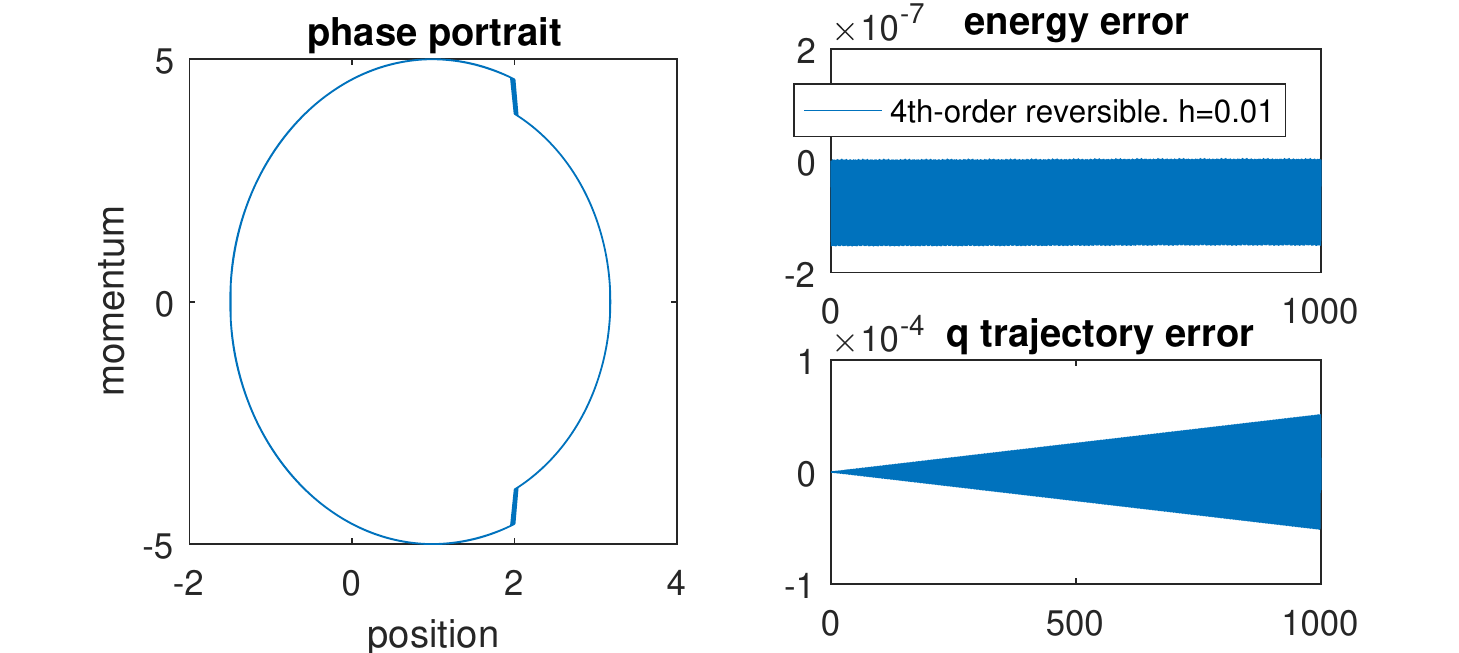}
	\end{subfigure}
    \begin{subfigure}[t]{0.49\textwidth}
		\includegraphics[width=\textwidth]{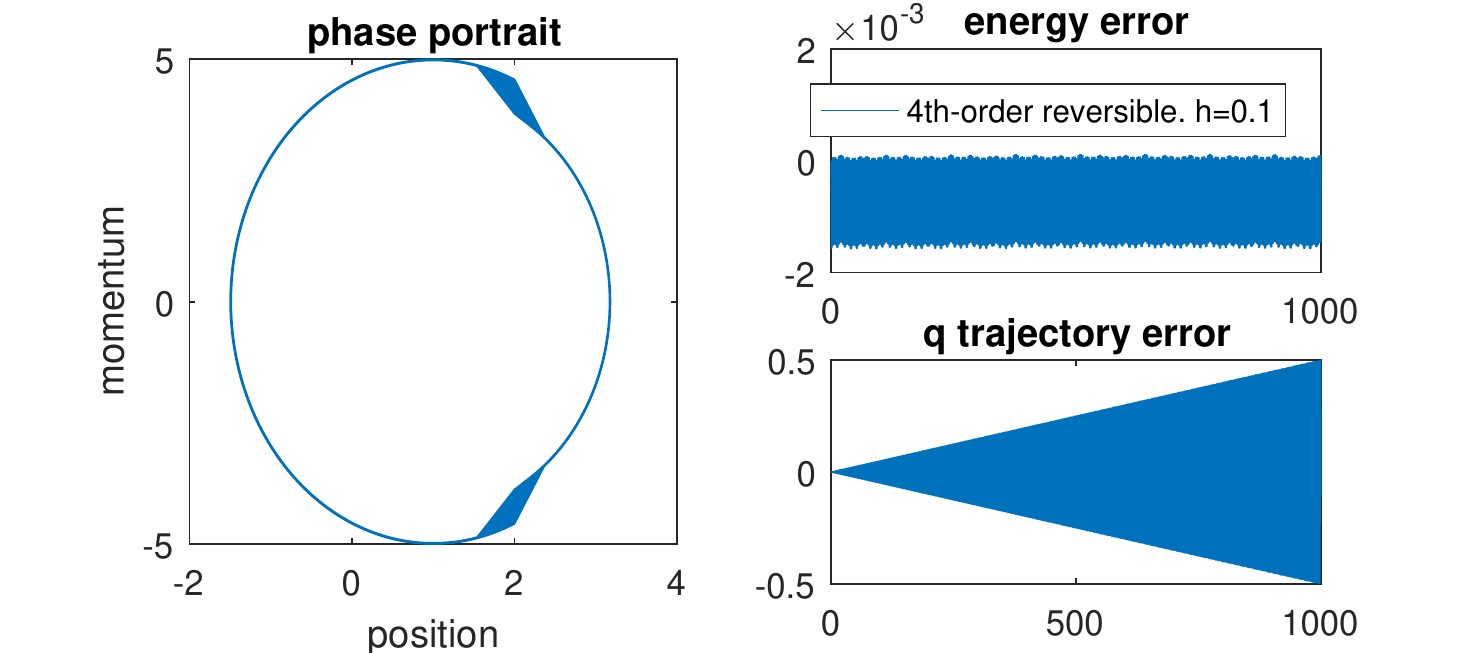}
	\end{subfigure}
    \begin{subfigure}[t]{0.49\textwidth}
		\includegraphics[width=\textwidth]{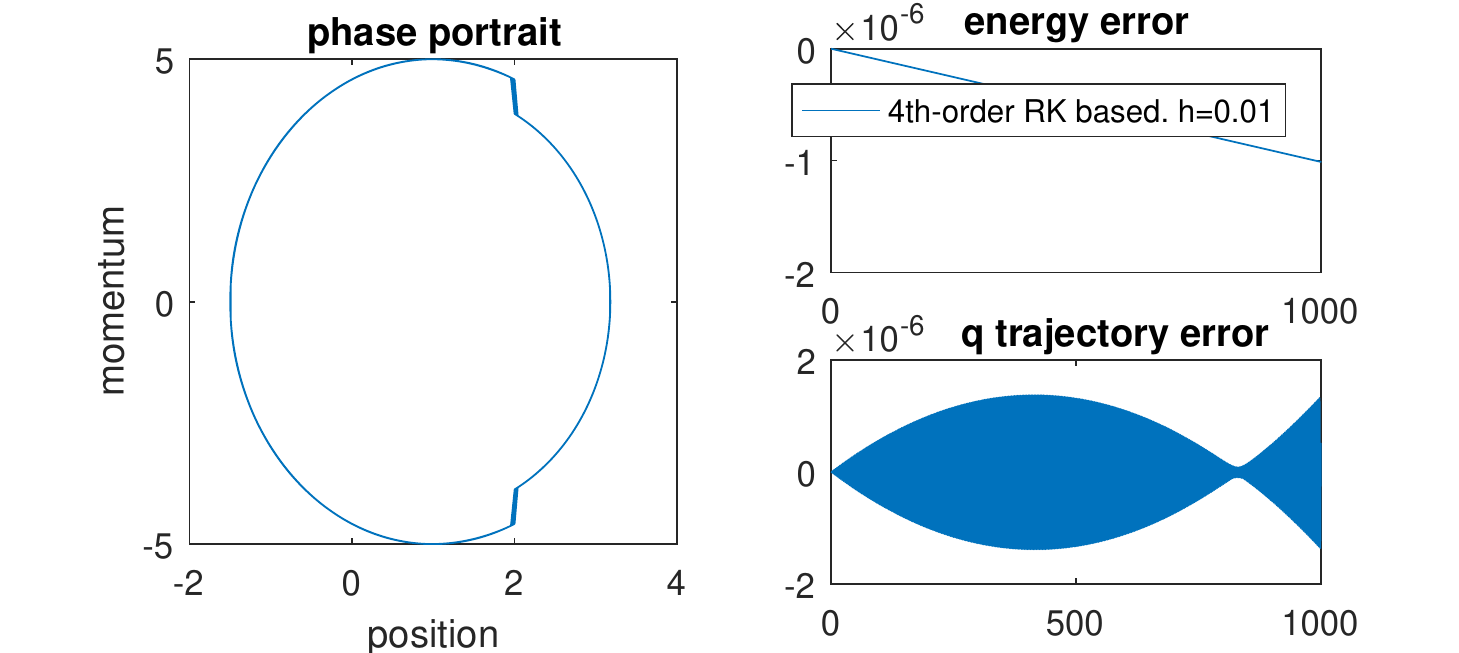}
	\end{subfigure}
    \begin{subfigure}[t]{0.49\textwidth}
		\includegraphics[width=\textwidth]{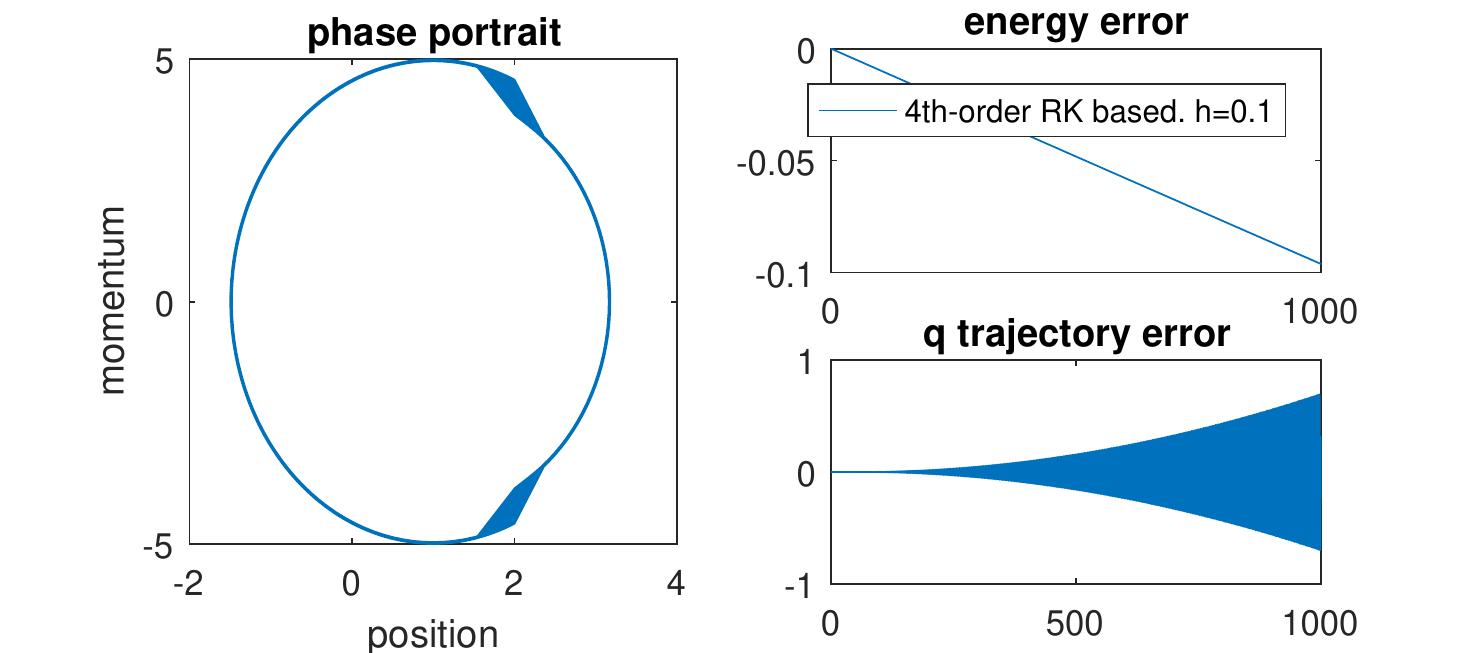}
	\end{subfigure}
    \begin{subfigure}[t]{0.49\textwidth}
		\includegraphics[width=\textwidth]{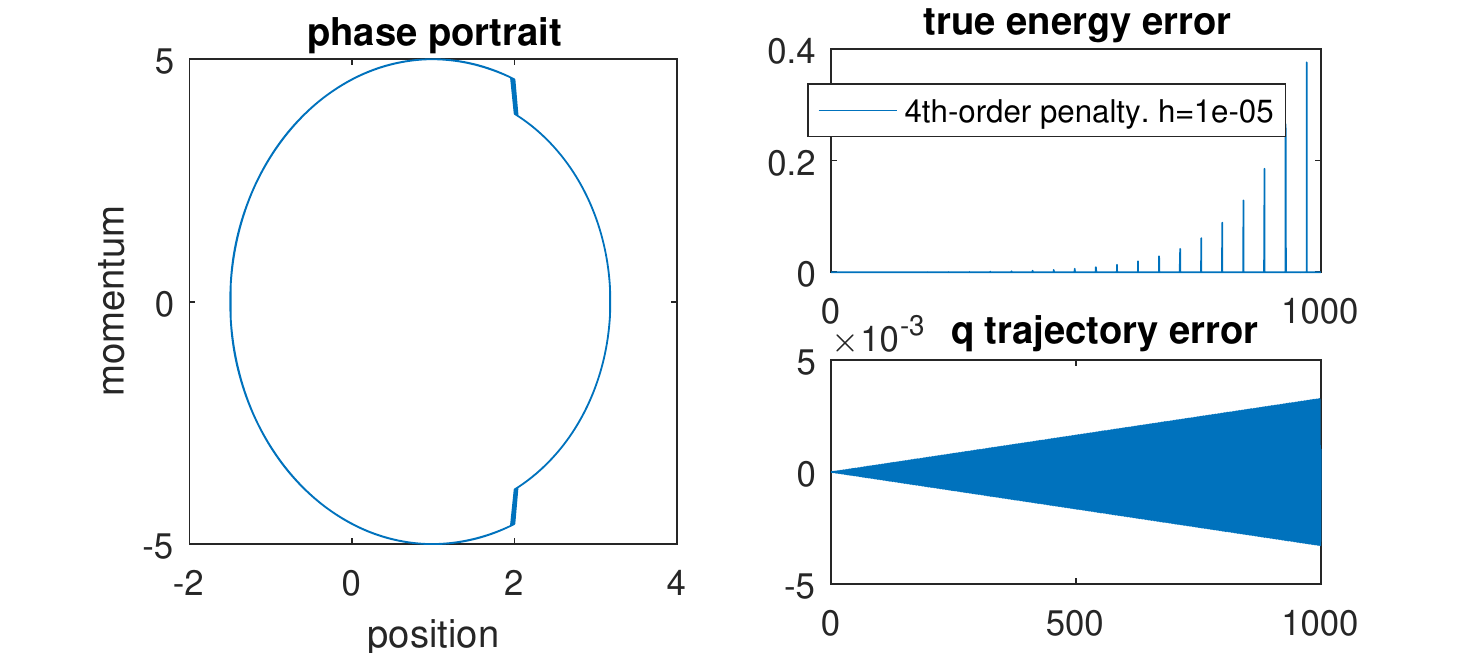}
	\end{subfigure}
    \begin{subfigure}[t]{0.49\textwidth}
		\includegraphics[width=\textwidth]{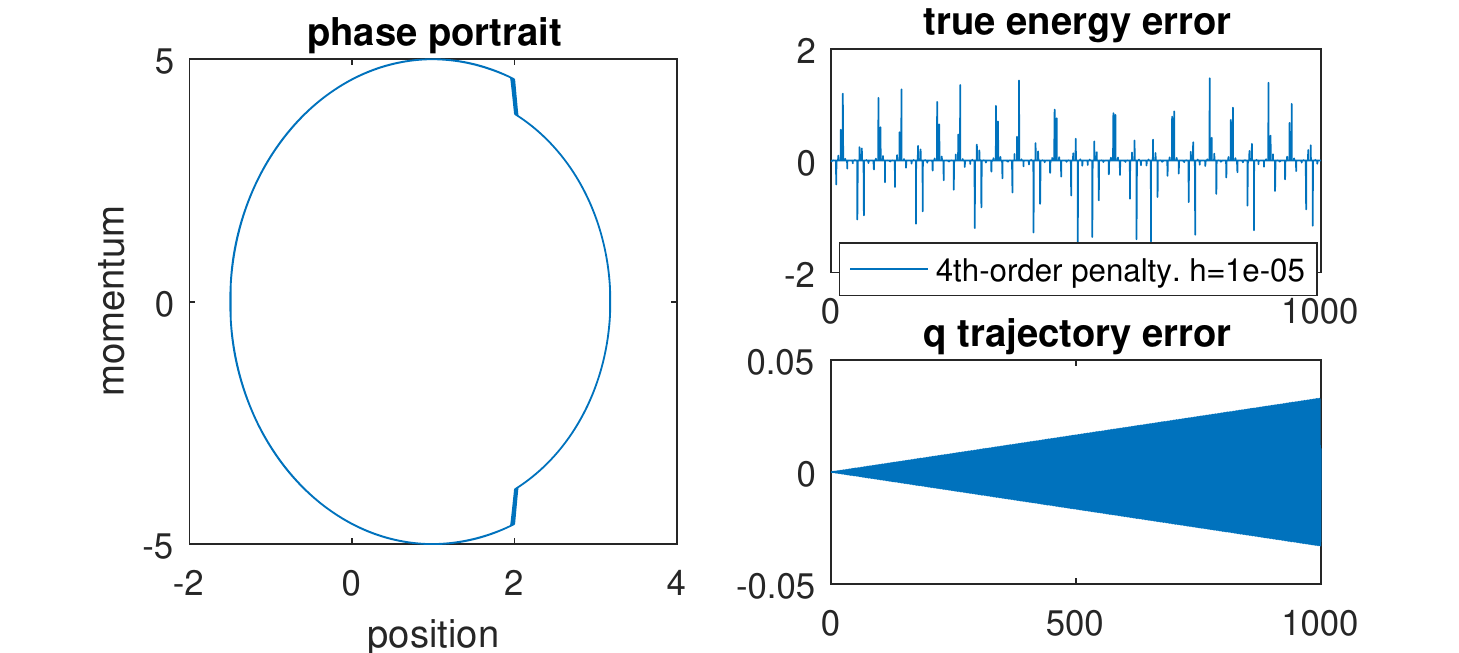}
	\end{subfigure}
	\caption{A benchmark problem: potential = quadratic + step function. First 4 rows: respectively, 1st-order symplectic, 1st-order non-symplectic, 4th-order reversible, 4th-order \emph{ir}reversible (Runge-Kutta based); left: $h=0.01$, right: $h=0.1$. Last row: penalty method, i.e., 4th-order symplectic integration of a regularized Hamiltonian; left: $\alpha=10^4$, $h=10^{-5}$, right: $\alpha=10^3$, $h=10^{-5}$.}
	\label{fig:quadraticProb}
\end{figure}

Fig. \ref{fig:quadraticProb} compares the performances of our 1st-order symplectic Integrator \ref{int_1stOrderSymplectic} (Sec. \ref{sec_1stOrderSympl}), a non-symplectic version of this 1st-order method, our 4th-order reversible Integrator \ref{int_lstOrder} (based on a 4th-order reversible smooth integrator) (Sec. \ref{sec_highOrderNonSympl}), an irreversible version of this 4th-order method (based on a smooth integrator of 4th-order Runge-Kutta), and the penalty method (based on a regularized Hamiltonian and a 4th-order reversible symplectic integration of this smooth Hamiltonian). The 3rd-order symplectic integrator in Sec. \ref{sec_3rdOrderSympl} is excluded because it gives the exact solution for this example. The exact solution is periodic with period $\approx 3$, and the comparison is over $T=10^3$ which should be considered as a long time.

The penalty method simulates a regularized smooth penalty Hamiltonian
\begin{equation}
	H(q,p)=\|p\|_2^2/2+U(q)+\Delta V \frac{1}{1+\exp\big(-\alpha(q-q_{jump})\big)}.
	\label{eq_regularizedHamiltonianNumerics1}
\end{equation}
The simulation uses 4th-order symplectic integrator based on triple jump (see e.g., \cite{Hairer06}).

The non-symplectic 1st-order integrator used is simply a forward Euler type, with one $h$-step update given by
\[
	[q,p] \mapsto [q,p]+(\phi_1^{h/2}-id)[q,p] + (\phi_2^h-id)[q,p] + (\phi_1^{h/2}-id)[q,p],
\]
where $\phi_1$ is given by \eqref{eq_phi1}, $\phi_2$ is given by (\ref{eq_phi2hittingTime}--\ref{eq_phi2postImpact}), and $id$ is the identity map. As a reminder and a comparison, the symplectic 1st-order integrator used here is $[q,p] \mapsto \phi_1^{h/2} \circ \phi_2^h \circ \phi_1^{h/2} [q,p]$.

\paragraph{Results.} The left half of Fig.\ref{fig:quadraticProb} shows that the 1st-order symplectic method and the 4th-order reversible method exhibit linear growth of error, and almost no drift in energy but only fluctuations. Both are similar to that of the symplectic/reversible integration of smooth integrable systems. On the contrast, the 1st-order non-symplectic method has too much artificial energy injected due to the numerics, and the 4th-order \emph{ir}reversible method (Runge-Kutta based) has artificial energy dissipation, although the amount is small due to small $h$, high-order, and $T$ not too large. More on long time performance will follow.

Comparing the first 4 rows of the left and right halves of Fig.\ref{fig:quadraticProb}, which differ by different step sizes, one sees consistency with the claimed order of each method. More on convergence order will follow.

The 5th row shows that the penalty method, when used with a sufficiently small $h$, such as $o(1/\alpha)$, has error that is only 1st order in $1/\alpha$.

\begin{figure}
	\centering
	\includegraphics[width=\textwidth]{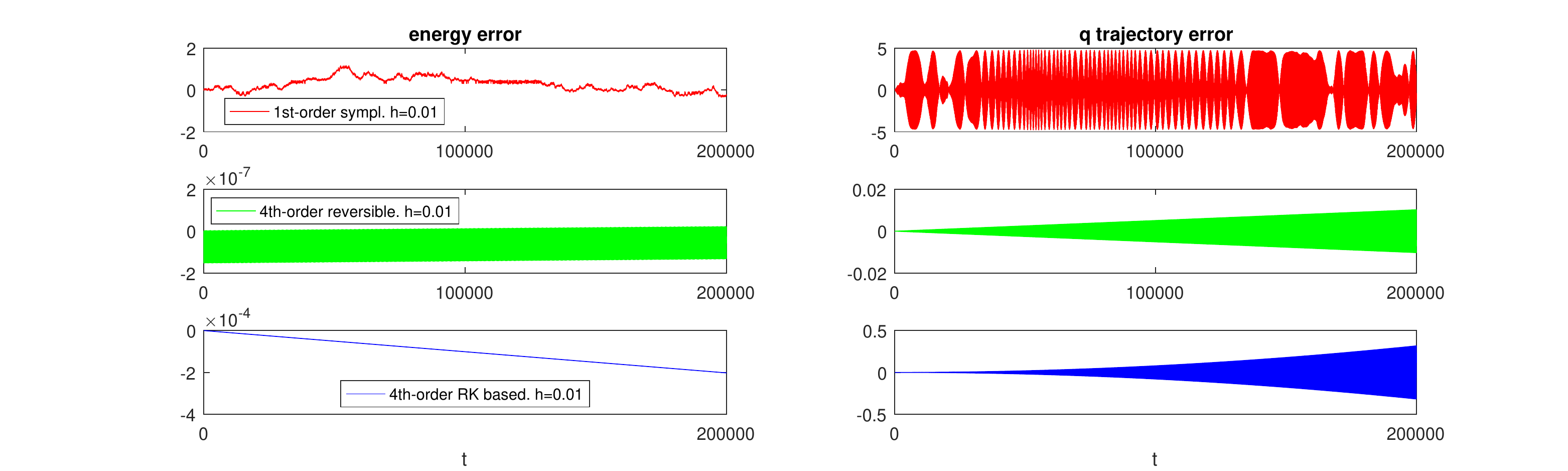}
	\caption{Super long time performances of proposed methods. 3 rows: respectively, 1st-order symplectic, 4th-order reversible, 4th-order irreversible. Penalty method unaffordable.}
	\label{fig:quadraticProbLongTime}
\end{figure}

To further study these observations, performances over even longer time span ($T=2\times 10^5$) are investigated in Fig.\ref{fig:quadraticProbLongTime}. We see rather irregular, however bounded global error of the 1st-order symplectic method. This should not be surprising as the classical Hamiltonian backward error analysis no longer applies due to discontinuity. The 4th-order irreversible (Runge-Kutta based) method artificially dissipates energy, and its solution error changes like the traditional exponential error growth of non-symplectic methods for smooth problems. The 4th-order reversible (symplectic integrator based) method seems to exhibit linear error growth; however, we note a small but definite drift in its energy error, which is not solely oscillatory. We hope that a symplectic counterpart would not have this drift, but its design remains an open problem.

\begin{figure}
	\centering
	\includegraphics[width=0.5\textwidth]{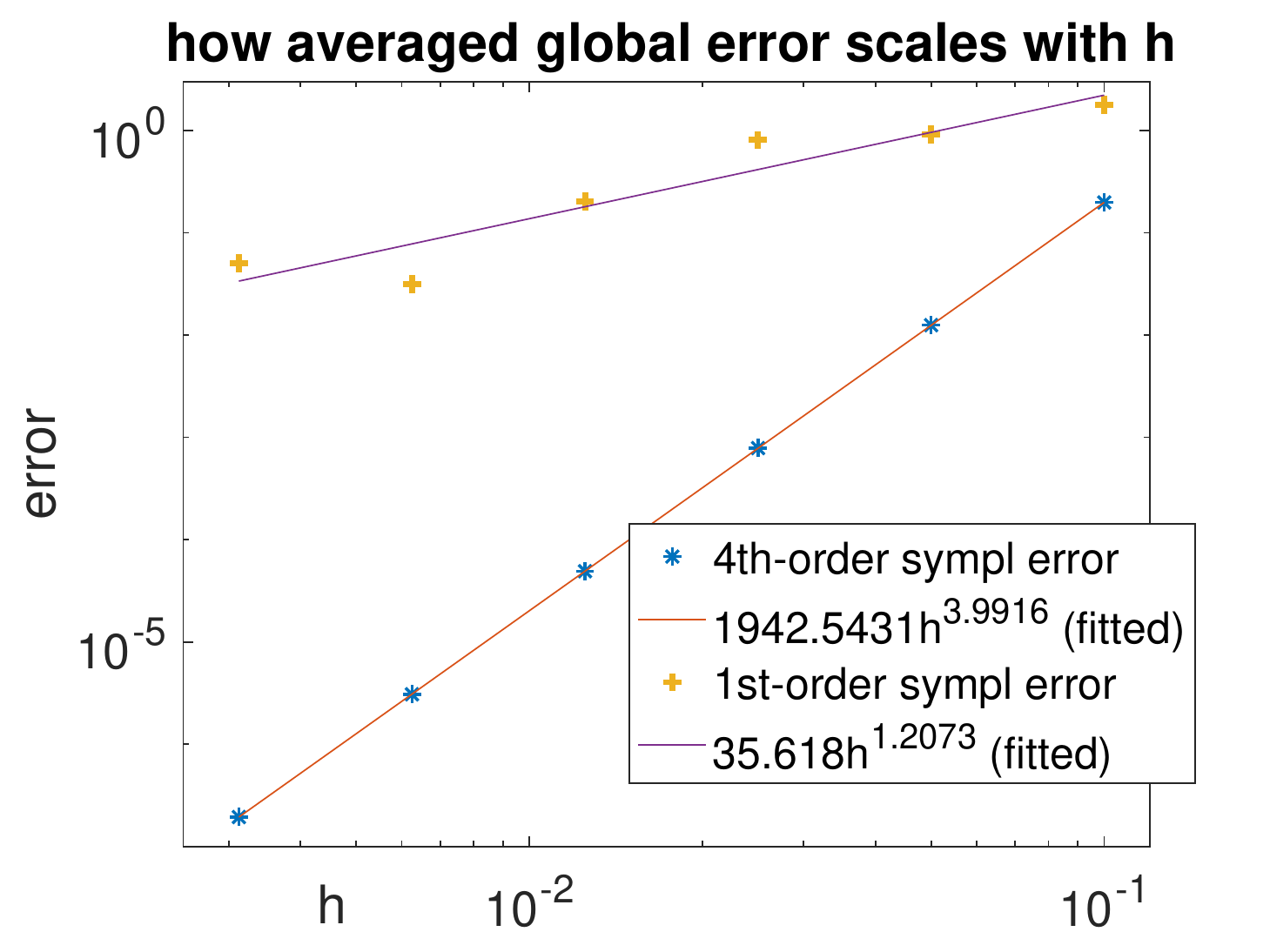}
	\caption{Verifications of orders of 1st-order symplectic Integrator \ref{int_1stOrderSymplectic} and 4th-order reversible Integrator \ref{int_lstOrder} (based on a symplectic reversible smooth integrator). Root-mean-square errors are computed as $\sqrt{1/(T/h+1)\sum_{i=0}^{T/h} \left(q^{num}_i-q^{exact}(ih)\right)^2}$, $T=1000$.}
	\label{fig:quadraticProbConvOrder}
\end{figure}

Fig.\ref{fig:quadraticProbConvOrder} confirms that our 1st-order symplectic integrator and 4th-order reversible integrator are indeed 1st and 4th order (note time points of \emph{impact} are few and therefore negligible after the averaging).

\subsection{Improved accuracy when there is only one linear interface: a nonlinear example}
\label{sec_numerics3rdOrder}

Let's now consider a problem with smooth and jump potentials, respectively,
\[
	U(q)=(q-q_c)^4/12,	\qquad 
	V(q) = \begin{cases}
			\Delta V, \qquad & q>q_{jump} \\
			0, \qquad & q<q_{jump} \\
			\text{undefined}, \qquad & q=q_{jump}
		   \end{cases}.
\]
Due to the nonlinearity created by $U$, no exact solution is available as a benchmark to compare against. However, as there is only one linear interface, the high-order symplectic method in Sec. \ref{sec_3rdOrderSympl} applies.

\begin{figure}
	\centering
	\includegraphics[width=\textwidth]{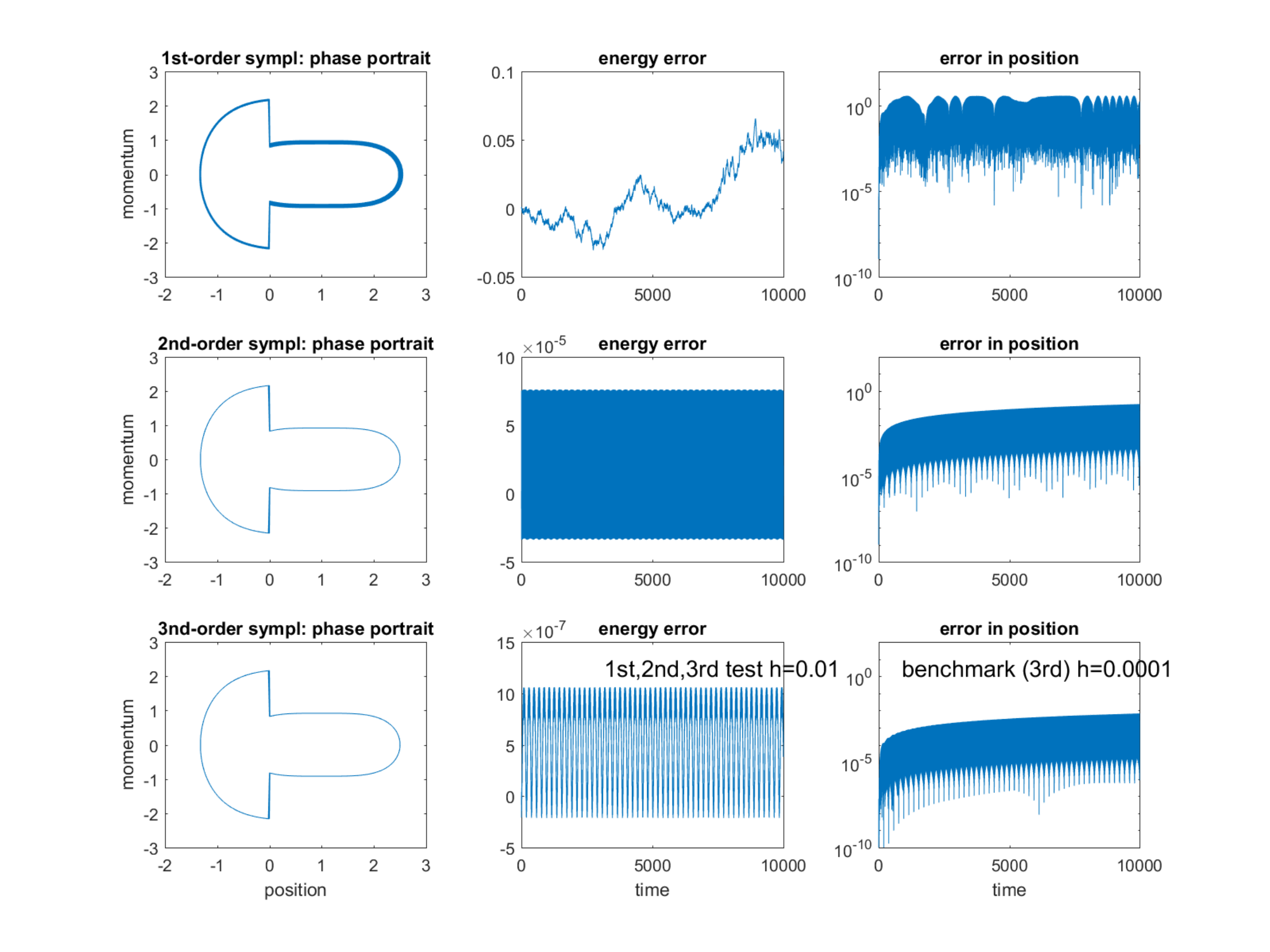}
	\caption{Long time performances of 1st-, 2nd-, and 3rd-order symplectic integrators for a nonlinear problem with one linear interface.}
	\label{fig:quarticLongTime}
\end{figure}

Fig. \ref{fig:quarticLongTime} compares the performances of a 3rd-order Integrator \ref{int_3rdOrderSymplectic}, a 2nd-order integrator (Remark \ref{int_2ndOrderSymplectic}), and a 1st-order Integrator \ref{int_1stOrderSymplectic}, all symplectic. The results indicate consistency with the claimed orders of the methods (Fig. \ref{fig:quarticConvOrder} further confirms this), and the 2nd- and 3rd-order versions have much more regular long time energy behaviors (note: all three are reversible!)

\begin{figure}
	\centering
	\includegraphics[width=0.5\textwidth]{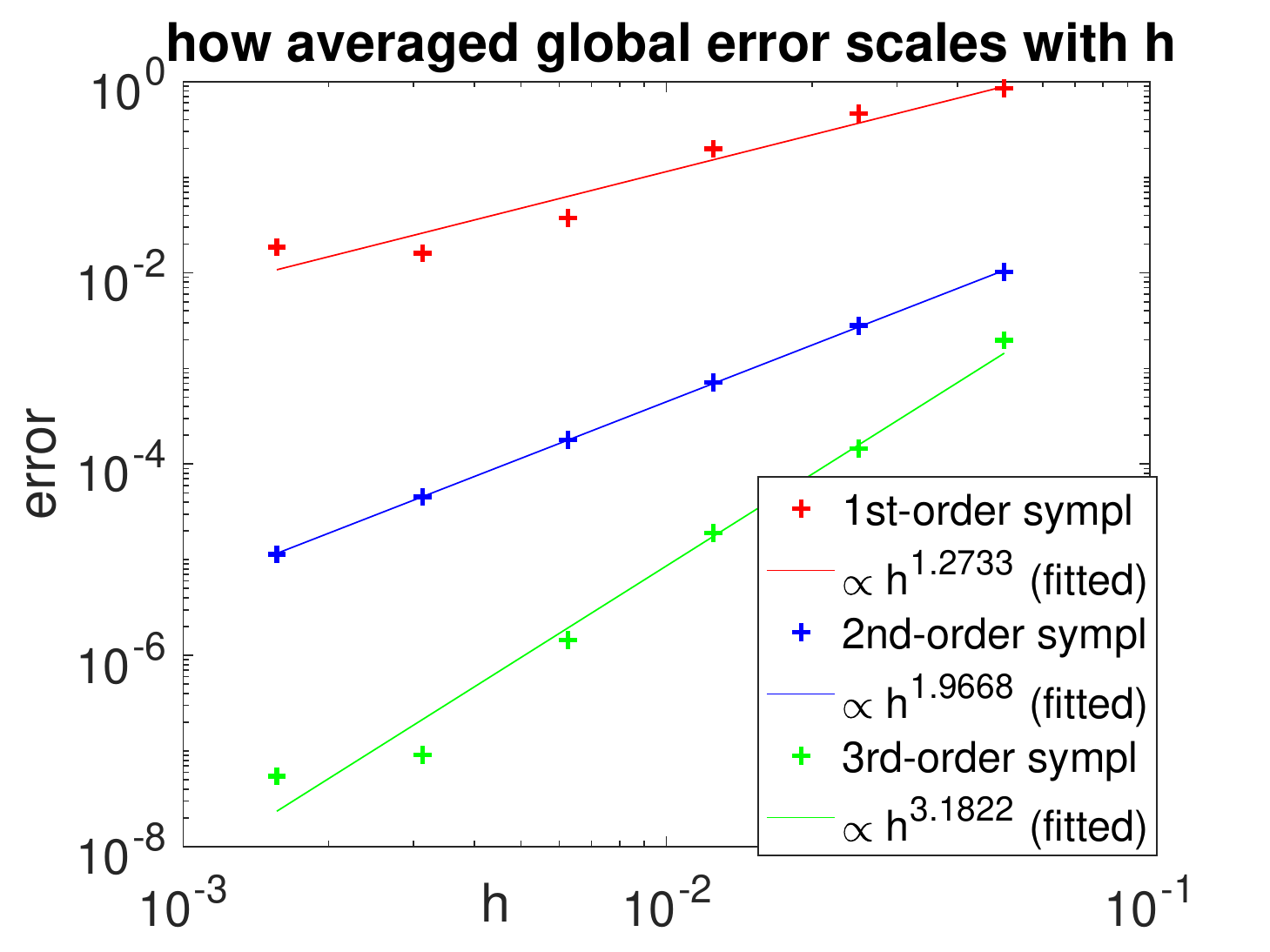}
	\caption{Verifications of the orders of 1st-, 2nd-, and 3rd-order symplectic methods (Integrator \ref{int_1stOrderSymplectic}, that in Rmk. \ref{int_2ndOrderSymplectic}, and Integrator \ref{int_3rdOrderSymplectic}). Averaged global errors are root-mean-square errors, computed as $\sqrt{1/(T/h+1)\sum_{i=0}^{T/h} \left(q^{num}_i-q^{benchmark}(ih)\right)^2}$, $T=100$.}
	\label{fig:quarticConvOrder}
\end{figure}

In these experiments, $\Delta V=2$, $q_{jump}=0$, $q_c=1$. Trajectory errors are computed by comparing against a tiny step-sized ($h=10^{-5}$) 3rd-order (Integrator \ref{int_3rdOrderSymplectic}) simulation.



\subsection{An example in which the interface is given by a level set; conservation of momentum map}
\label{sec_numericsModifiedKepler}


\paragraph{Setup.} We now consider an example in which the discontinuous Hamiltonian has a symmetry. For continuous Hamiltonians, this would imply the conservation of a corresponding momentum map, due to Noether's theorem \cite{Noether:18}. Moreover, a symplectic discretization of the continuous Hamiltonian system can inherit this conservation under nontrivial but reasonable conditions (see \cite{MaWe:01} Chap. 1.3.3 and 1.4.2 for details). Unfortunately, the analogous results for discontinuous Hamiltonians are currently unknown. For our specific example, however, the exact solution would still have a symmetry-based conservation law in addition to energy conservation, and this section investigates whether symplectic Integrator \ref{int_1stOrderSymplectic} numerically captures this conservation too.

This example has 2 degrees of freedom and significant nonlinearity. The smooth potential is gravitational, $U(q)=-1/\|q\|_2$. However, the solution does not follow the classical 1-body dynamics which corresponds to Keplerian orbits, as there is an additional nonsmooth potential
\[
	V(q)=\begin{cases} 
			0, \quad & \|q\|<r_{jump} \\
			\Delta V, \quad & \|q\|>r_{jump}
		\end{cases}.
\]
Obviously here the discontinuous interface is nonlinear, and representable by the zero level set of function $\|q\|-r_{jump}$.

The Hamiltonian $\|p\|_2^2/2+U(q)+V(q)$ is invariant under rotations in the plane, and it is not hard to see its exact solution conserves the angular momentum $L:=p\times q = p_1 q_2 - p_2 q_1$ as \emph{impact} only changes the radial component of $p$. 

Meanwhile, note that although the Hamiltonian is invariant under rotations, its trajectory is not. Even without the jump discontinuity, the solution as a Keplerian orbit is not a circle, unless its initial condition is special enough to lead to a zero eccentricity.

\begin{figure}
	\centering
    \begin{subfigure}[t]{0.49\textwidth}
		\includegraphics[width=\textwidth]{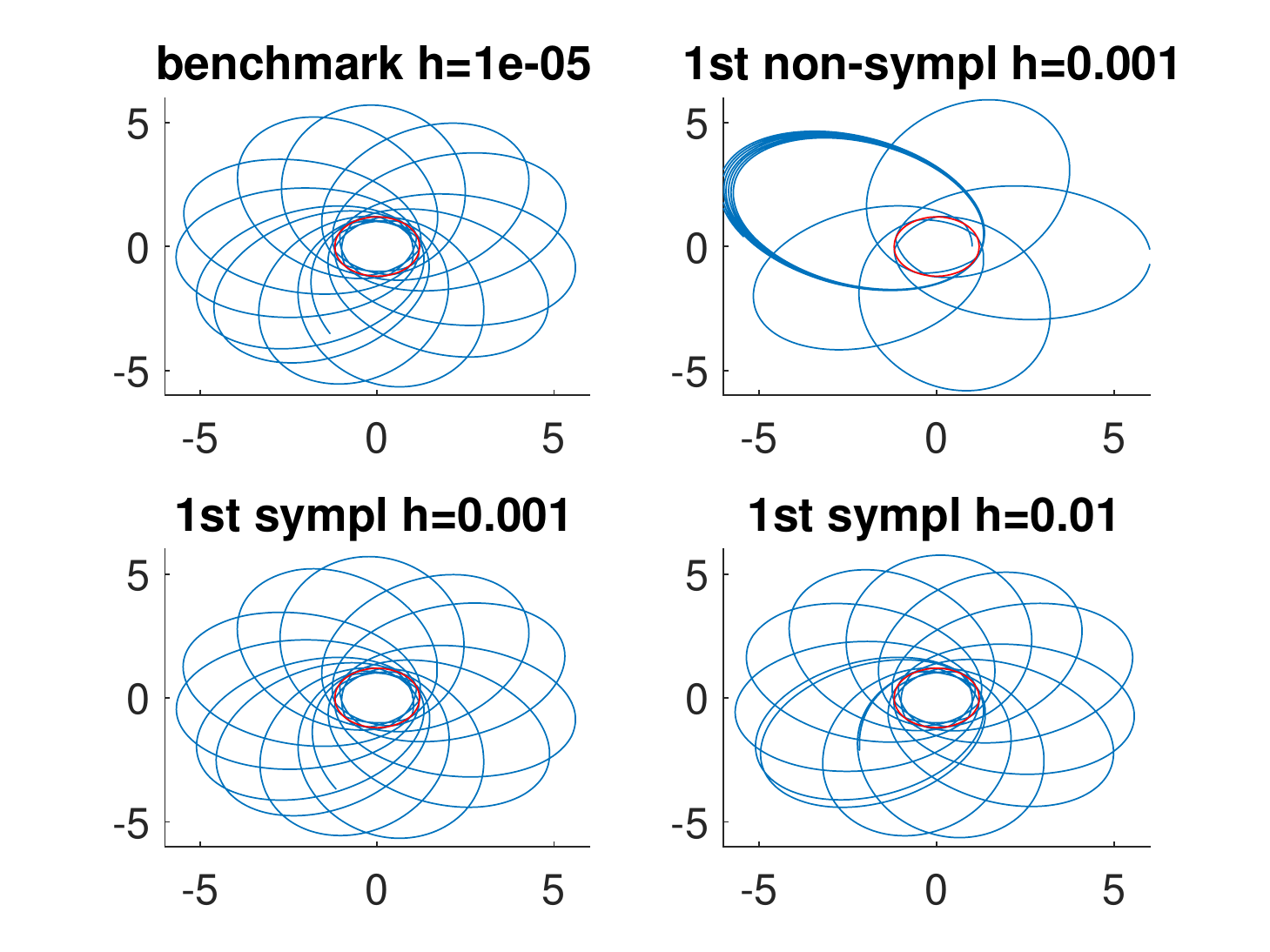}
		\caption{orbit projected to the $[q_1,q_2]$ plane} 
		\label{fig:angularMomentum:orbit}
	\end{subfigure}
    \begin{subfigure}[t]{0.49\textwidth}
		\includegraphics[width=\textwidth]{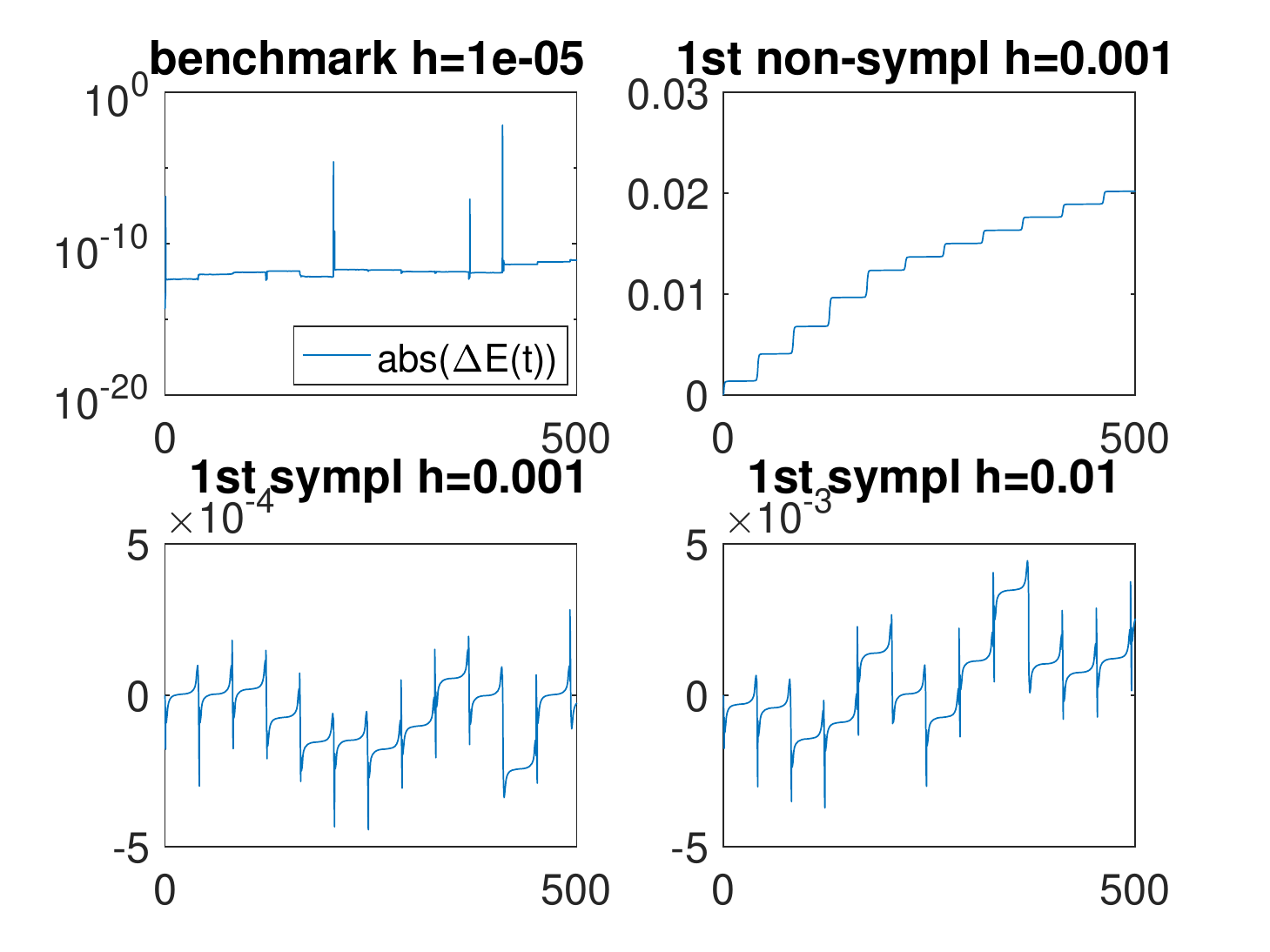}
		\caption{energy error}
		\label{fig:angularMomentum:energy}
	\end{subfigure}
    \begin{subfigure}[t]{0.49\textwidth}
		\includegraphics[width=\textwidth]{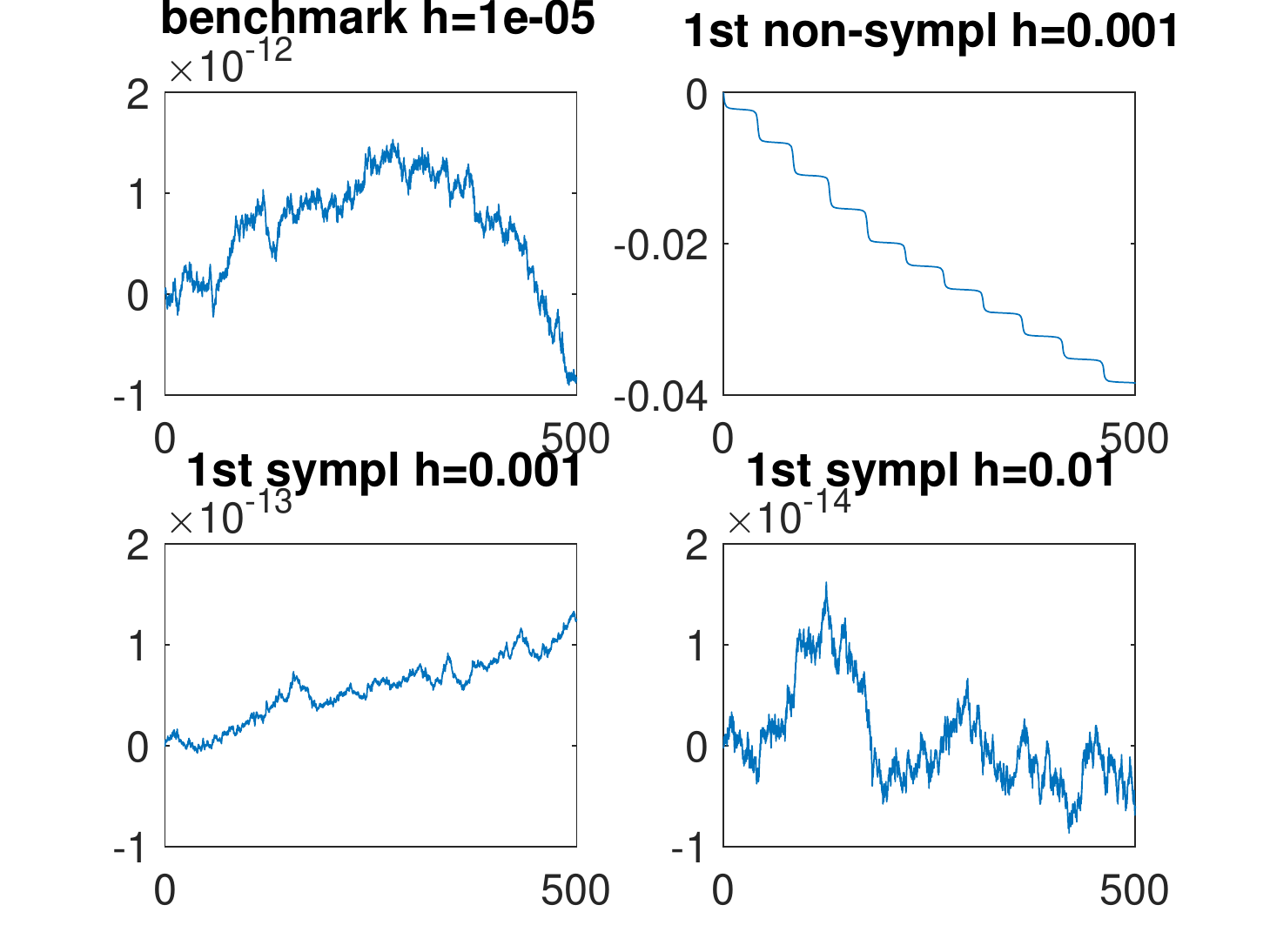}
		\caption{angular momentum error}
		\label{fig:angularMomentum:momentumMap}
	\end{subfigure}
    \begin{subfigure}[t]{0.49\textwidth}
		\includegraphics[width=\textwidth]{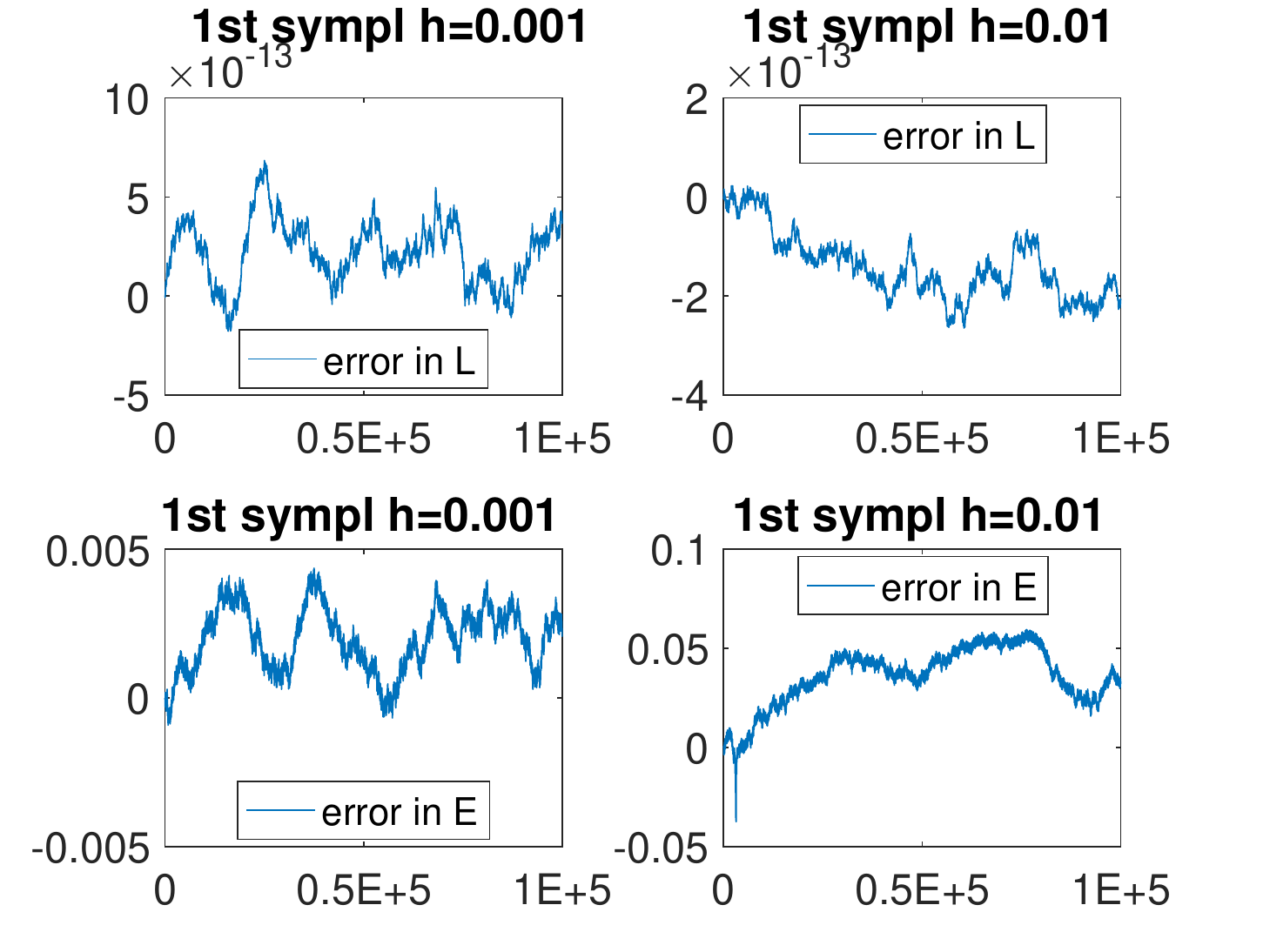}
		\caption{errors in angular momentum L and energy E over super long time; note these are unavailable for the benchmark method, and blowing up for the nonsymplectic method}
		\label{fig:angularMomentum:longTime}
	\end{subfigure}
	\caption{Conservation of angular momentum in the presence of rotational symmetry}
	\label{fig:angularMomentum}
\end{figure}

Fig.\ref{fig:angularMomentum} compares our 1st-order symplectic simulation (by Integrator \ref{int_1stOrderSymplectic}) with a benchmark solution that uses $\sim 100\times$ computational cost, as well as a nonsymplectic version of Integrator \ref{int_1stOrderSymplectic}. $\Delta V=0.125$, $r_{jump}=1.2$, $q(0)=[1,0], p(0)=[0,1.4]$. 

The nonsymplectic version used is simply a forward Euler type, with one $h$-step update given by
\[
	[q,p] \mapsto [q,p]+(\phi_1^h-id)[q,p] + (\phi_2^h-id)[q,p],
\]
where $\phi_1$ is given by \eqref{eq_phi1}, $\phi_2$ is given by (\ref{eq_phi2hittingTime}--\ref{eq_phi2postImpact}), and $id$ is the identity map. For a fair comparison, the symplectic version used here is an irreversible variation (based on Lie-Trotter splitting instead of Strang splitting), with one $h$-step update given by.
\[
	[q,p] \mapsto \phi_1^h \circ \phi_2^h [q,p].
\]

The benchmark solution was generated by fine symplectic simulation of a regularized smooth penalty Hamiltonian
\begin{equation}
	H(q,p)=\|p\|_2^2/2+U(q)+\Delta V \frac{1}{1+\exp\big(-\alpha(\|q\|-r_{jump})\big)}.
	\label{eq_regularizedHamiltonianNumerics3}
\end{equation}
This simulation uses 4th-order symplectic integrator based on triple jump (see e.g., \cite{Hairer06}) with $h=10^{-5}$, with penalty parameter $\alpha=10^5$. Both parameters $\alpha$ and $h$ were tuned to ensure high precision with lowest possible computational cost.

\paragraph{Results.}
Fig. \ref{fig:angularMomentum:orbit} illustrates the orbits and our method agrees well with the benchmark but uses a $100\times$ larger stepsize, and that even if it uses a $1000\times$ larger stepsize, the long time error is still moderate (mainly due to accumulated phase error). For the purpose of visualizing the orbit, the simulation time is chosen to be $T=500$, which is relatively long, as one can see a nonsymplectic method gradually loses its energy and the particle eventually drops to an orbit with a large semi-major axis which no longer crosses the interface.

Fig. \ref{fig:angularMomentum:energy} and \ref{fig:angularMomentum:momentumMap} respectively plot how the energy and angular momentum, computed from the numerical solutions, deviates from their true values in time dependent ways. As expected, (i) the 1st-order symplectic method exhibits $\mathcal{O}(h)$ fluctuation in energy, while a nonsymplectic version accumulates error in energy; (ii) angular momentum is numerically conserved; the small error is due to limited machine precision, and $h=0.001$ gives more error than $h=0.01$ because it uses $10\times$ more steps, each of which induces a small arithmetic error. Fig. \ref{fig:angularMomentum:longTime} confirms that these deviations are truly bounded over super long time ($T=100,000$).

\subsection{Sauteed Mushroom: irregular interface geometry and complex dynamics (trapped or ergodic?)} 
\label{sec_numericsMushroom}

Finally, we demonstrate the capability of the proposed approach using an example where both the interface and the corresponding dynamics are complicated. Among all methods mentioned in this paper, only the adaptive Integrator \ref{int_adaptive} suits the investigation of the ergodic aspect of the dynamics, which requires accurate and affordable long-time simulation, because it can be $\geq 4$th-order and capable of capturing multiple \emph{impact}s in a short duration while still using a large step size. Note the purpose of this section changed a little bit, as we are shifting from demonstrating the correctness of the proposed method to using it as a tool that, {\it for the first time}, allows us to probe some hard problems and make conjectures.

More precisely, let's study a system that complicates the Bunimovich Mushroom, which is a classical example of Hamiltonian systems in divided phase space that `demonstrates a continuous transition from a completely chaotic system (stadium) to a completely integrable one (circle)' \cite{bunimovich2001mushrooms}. The specific mushroom we consider is a subset of $\mathbb{R}^2$, defined as
\[
	\mathcal{M} = \{(x,y) \mid x^2+y^2 \leq 2, y \geq 0\} \cup \{(x,y) \mid |x|\leq 1,|y|\leq 1, y\leq 0\}
\]
In the language of this paper, the classical Bunimovich Mushroom considers the discontinuous Hamiltonian dynamics of a particle, with initial condition inside $\mathcal{M}$, without any smooth potential (i.e., $U(q)=0$), and an infinite potential barrier at the mushroom boundary (i.e., $V(q)=0$ if $q\in \mathcal{M}^\circ$, $V(q)=+\infty$ if $q\in \mathcal{M}^c$, and undefined otherwise). The particle basically travels in straight line at constant speed until hitting the boundary, and then be reflected and travels as a free particle again until the next reflection, and the whole procedure repeats. Note the reflections can be arbitrarily frequent due to sharp corners (and hence our choice of an adaptive integrator). Among many beautiful results, one was the demonstration of that the phase space splits an integrable island and a chaotic sea \cite{bunimovich2001mushrooms}, and initial conditions in one region will not be able to percolate into the other region.

New to this paper is the addition of a nontrivial smooth potential. We are interested in how it could change the global dynamics. Specifically, consider the aforementioned jump potential $V$ and a smooth potential $U(q)=a((q_1-q^s_1)^4+(q_2-q^s_2)^4)/4$, where $a$ and vector $q^s$ are constant parameters, corresponding to a vectorial anharmonic attraction to $q^s$. 

In all experiments presented here, the sautee source is fixed at $q_s=[-0.5;-2]$, i.e., the left bottom of the mushroom. The initial condition is fixed as $q(0)=[1.5;0.2]$, $p(0)=[0;1]$, which is in the regular island of the classical mushroom (i.e., $a=0$).

\begin{figure}
	\centering
    \begin{subfigure}[t]{0.32\textwidth}
		\includegraphics[width=\textwidth]{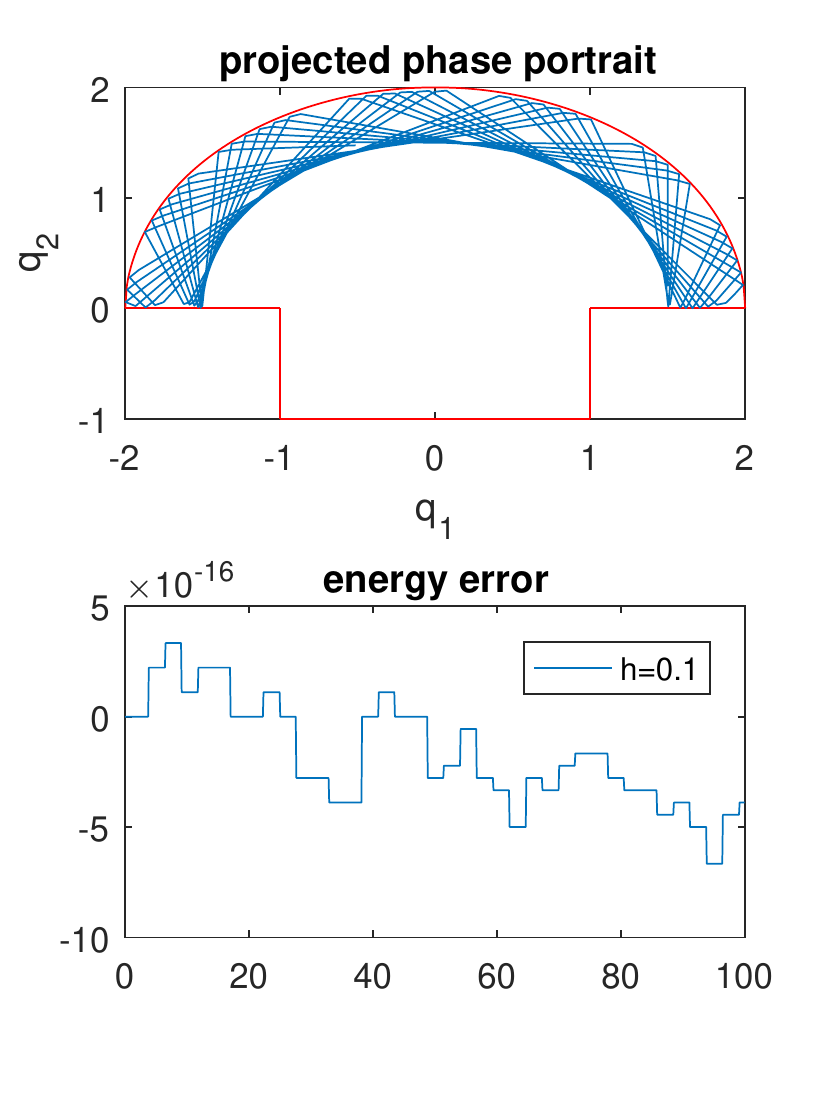}
		\caption{$a=0$, medium time} 
		\label{fig:mushroom:a0}
	\end{subfigure}
    \begin{subfigure}[t]{0.32\textwidth}
		\includegraphics[width=\textwidth]{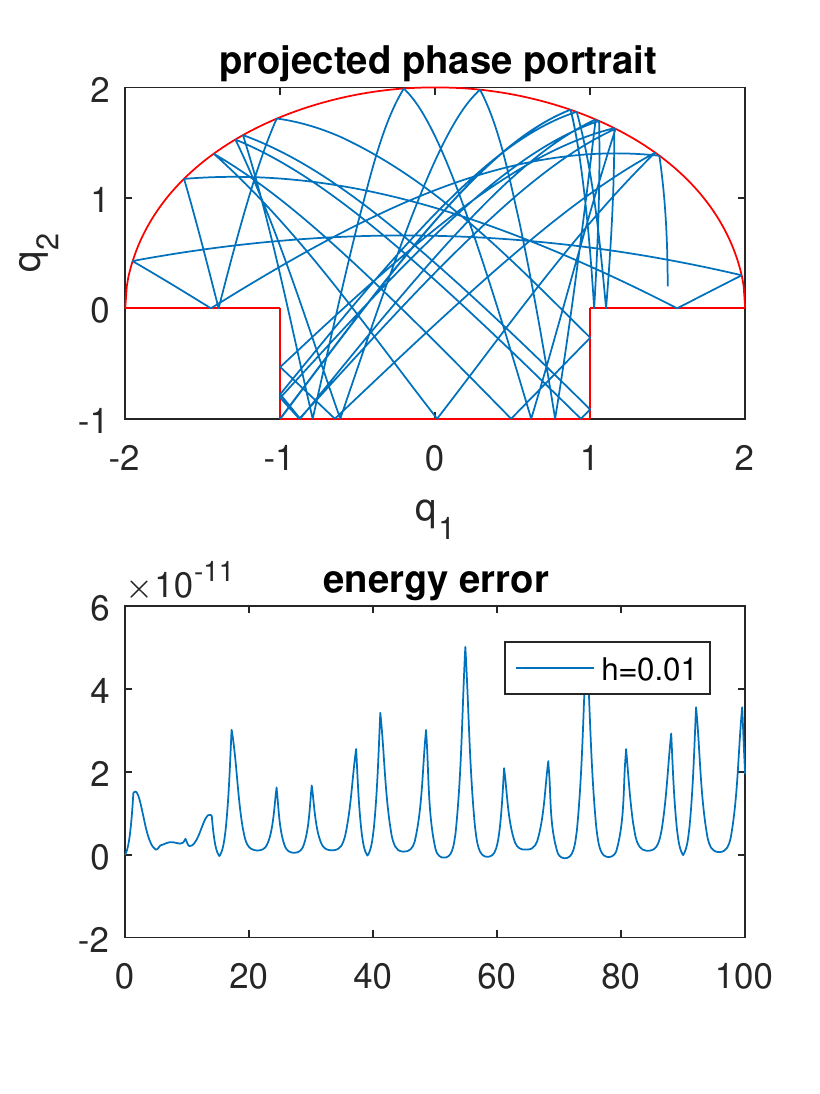}
		\caption{$a=0.008$, medium time}
	\end{subfigure}
    \begin{subfigure}[t]{0.32\textwidth}
		\includegraphics[width=\textwidth]{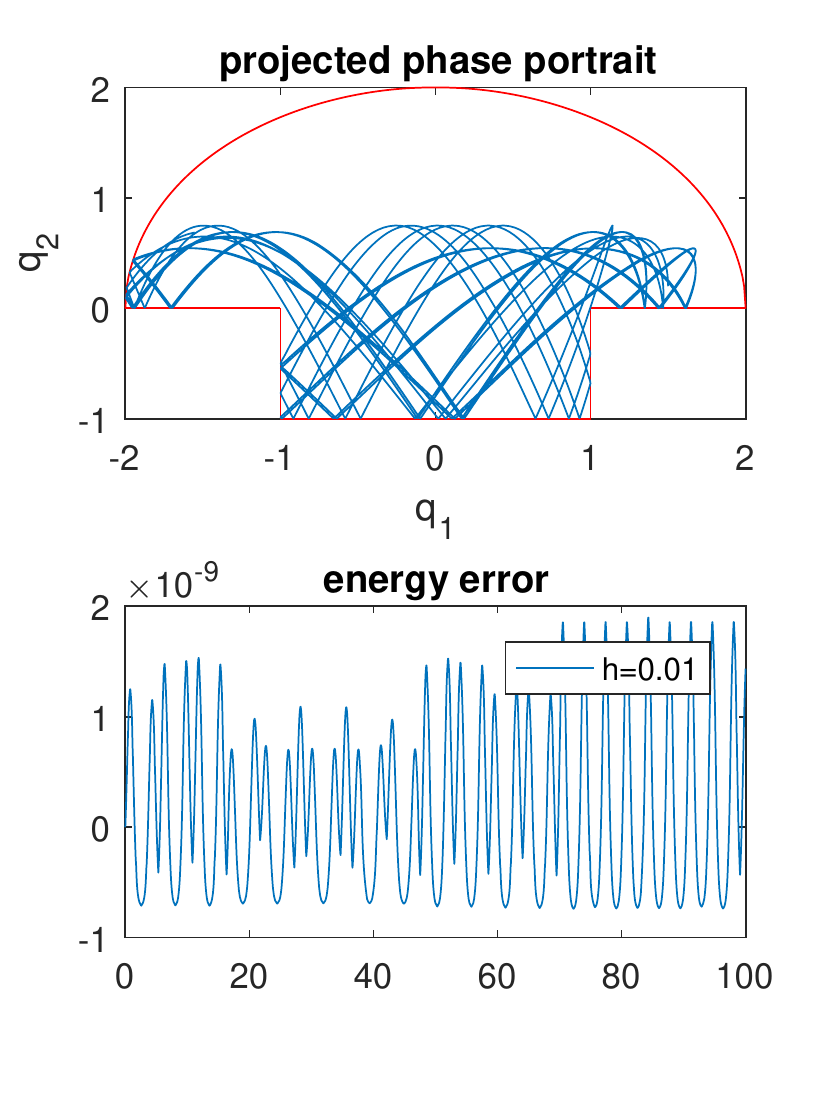}
		\caption{$a=0.08$, medium time} 
	\end{subfigure}
    \begin{subfigure}[t]{0.32\textwidth}
		\includegraphics[width=\textwidth]{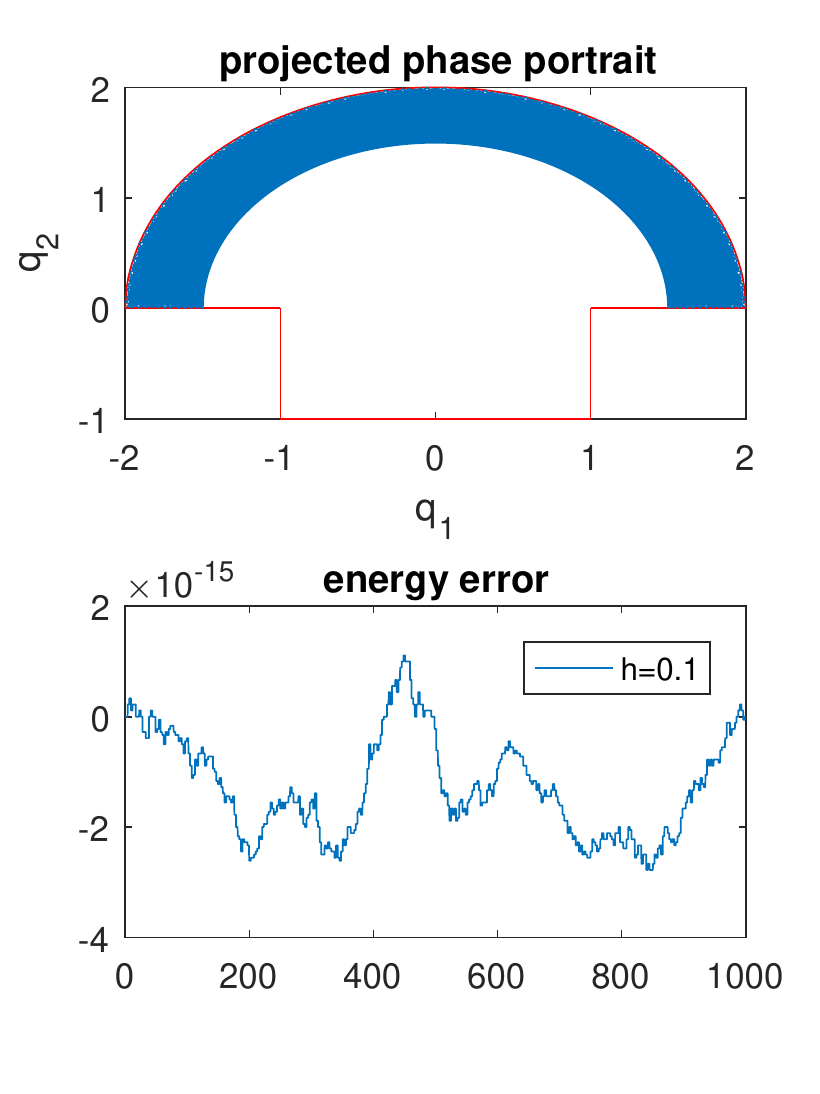}
		\caption{$a=0$, long time}
	\end{subfigure}
    \begin{subfigure}[t]{0.32\textwidth}
		\includegraphics[width=\textwidth]{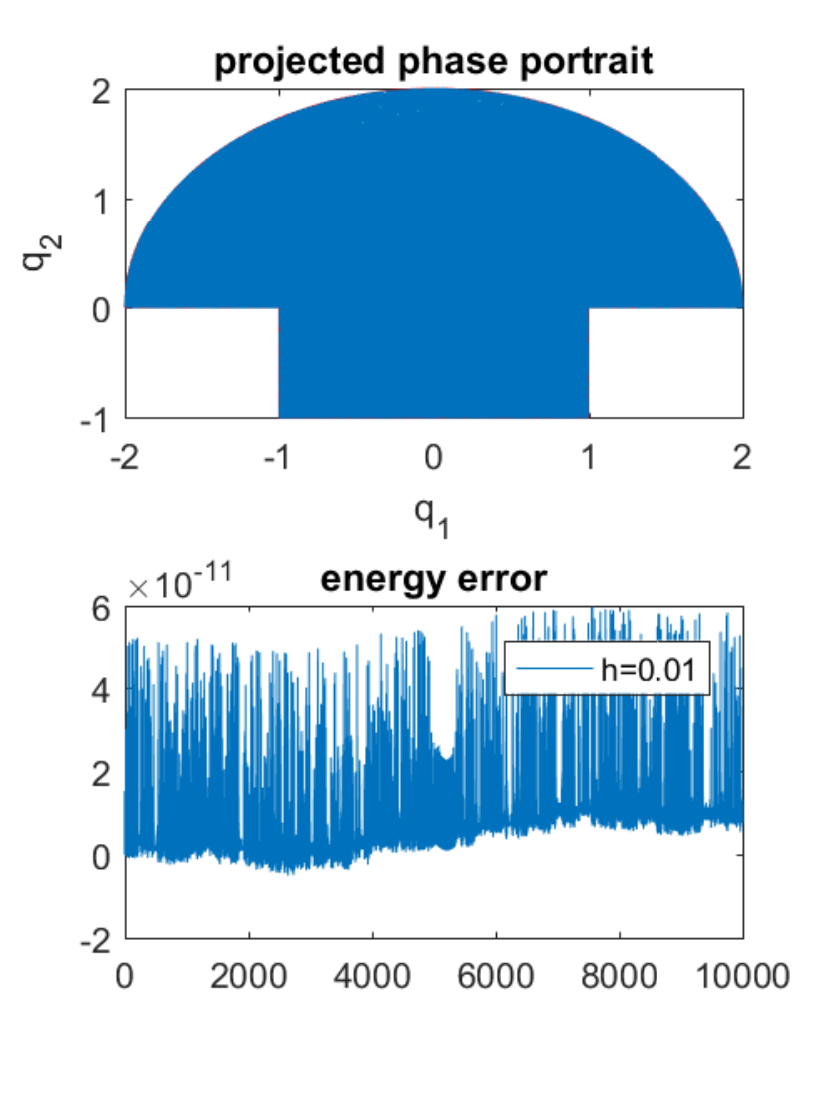}
		\caption{$a=0.008$, long time}
	\end{subfigure}
    \begin{subfigure}[t]{0.32\textwidth}
		\includegraphics[width=\textwidth]{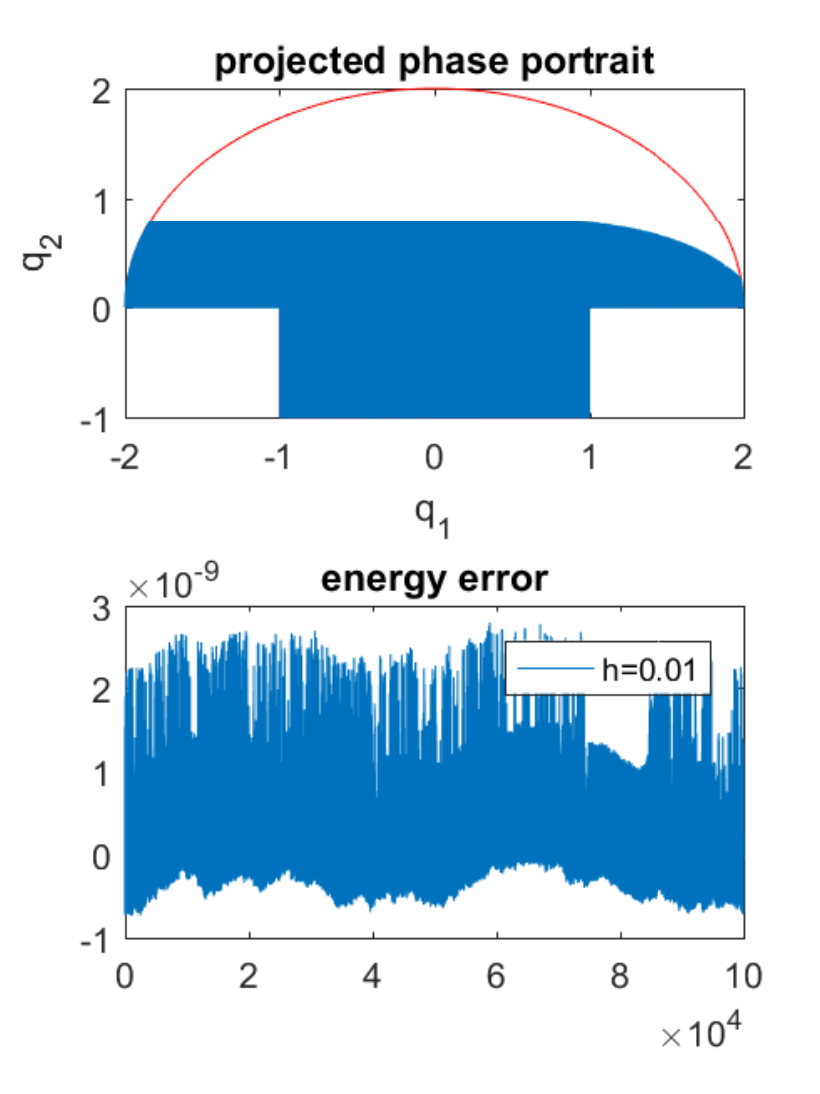}
		\caption{$a=0.08$, long time}
	\end{subfigure}
	\caption{Sauteed mushroom: switching between trapped and ergodic dynamics controlled by the sautee parameter $a$.}
	\label{fig:mushroom}
\end{figure}

Fig.\ref{fig:mushroom} shows distinct dynamics for different values of $a$ (short time simulations were provided in addition to long time ones for visualizing the dynamics). The same initial condition is used for all $a$ values. Although this initial condition corresponds to a regular island in the classical Bunimovich mushroom (Fig.\ref{fig:mushroom:a0}), when $a=0.008$ the dynamics appears to be chaotic and ergodic on the entire mushroom. When $a$ takes a larger value of $0.08$, however, it seems the dynamics is no longer ergodic any more, although still possibly chaotic, and the trajectory remains trapped in part of the mushroom.

Of course, these observations depend on the choice of the sautee source $q^s$ too.

Here both the legacy code $\psi$ used by Integrator \ref{int_adaptive} and  the integrator in the bisection method (Sec.\ref{sec_timeToImpact}) for estimating the time to \emph{impact} are the 4th-order symplectic integrator based on triple jump.


\section{Discussion and conclusion}
The accurate and efficient simulation of Hamiltonian mechanical systems with discontinuous potentials is an important problem. In fact, a special case, namely `impact/collision/contact integrators' for potentials with infinitely-high discontinuous barrier(s), has been extensively studied due to extensive applications in engineering and sciences.

The general case where jumps can be finite, however, appears to be insufficiently studied yet. To that end, this article,
along the line of \cite{Jin-Wen1, JinReview, JinWuHuang} in which the particle reflection and refraction at the interface are built into the dynamics, proposes four numerical methods, each with distinct applicability. As for general problems, the first method that we recommend to try (among the four plus the penalty method) is the adaptive high-order Integrator \ref{int_adaptive}. This is because of its robustness to complex interface geometry, together with the fact that whether/how symplecticity benefits long time accuracy is no longer clear (yet) in the discontinuous setting. This integrator already has, at least empirically, pleasant long time behaviors, and is computationally rather efficient too.

Several questions remain open. For example, (i) How to construct high-order symplectic integrators for general discontinuous potentials? Although we did obtain a 1st-order version for general problems, severe order-reduction from the classical continuous theory is encountered, and it is unclear if there is an order barrier or it is just that a higher-order explicit version remains to be developed. (ii) What would be the advantage(s) of having a symplectic method? Backward error analysis, if still applicable, needs to be completely revamped, and this includes both the modified equation/Hamiltonian theory and the error propagation analysis.

Moreover, the rich field of `impact/collision/contact integrators' has already developed a number of brilliant ideas, and we think many of them can be extrapolated to the more general setting in this article. For example, stabilization techniques may lead to further improved long term behaviors. Such (and more) explorations will be left as future work.

Other applications and extensions of these methods include geometrical optics, where waves can be {\it partially} 
transmitted and reflected \cite{Jin-Wen3} at interfaces, high frequency elastic  waves through interfaces \cite{Jin-Liao}, surface
hopping problems \cite{JinQiZhang} in computational chemistry, and quantum-classical couping algorithms \cite{JinNovak1, JinNovak2}. 

A side remark is that the sauteed mushroom (Sec.\ref{sec_numericsMushroom}) is definitely under-investigated in this article from a dynamical system perspective, but we hope it could demonstrate the applicability of our numerical integrator, and provoke thinking about its global dynamics and bifurcation in the future.

\section{Acknowledgment}
MT is thankful for the partial support by NSF DMS-1847802 and ECCS-1936776. SJ was supported by NSFC grant No. 12031013 and by the Strategic Priority Research 
Program of Chinese Academy of Sciences Grant No. XDA25010404.

\newpage
\section{Appendix}
\subsection{Why does Strang splitting no longer produce a 2nd-order method}
\label{appd_order1stOrderSymplectic}
This section will give an example for which Integrator \ref{int_1stOrderSymplectic} does not have a 3rd-order local truncation error, even though it is a time-reversible method constructed via symmetric Strang splitting (which is guaranteed to have a 3rd-order truncation error in the smooth case).

Consider the quadratic problem given in Section \ref{sec_quadraticPotentialSln} and denote by $q,p$ the current position and momentum. Assume $q=q_\text{jump}-C h$ for some bounded constant $C>0$, and $p>0$ is sufficiently large, so that an \emph{impact} will happen in $h$-time and the interface crossing will be a refraction. In this case, the exact solution after $h$-time, $Q,P$, is given by
\begin{align*}
	\hat{p} &= \sqrt{\omega^2 (q-q_\text{off})^2 + p^2 - \omega^2 (q_\text{jump}-q_\text{off})^2} \\
	t &= \big( 2\pi - \text{atan2}(\hat{p}/\omega, q_\text{jump}-q_\text{off}) + \text{atan2}(p/\omega, q-q_\text{off}) \big) / \omega \\
	\bar{p} &= \sqrt{\omega^2 (q-q_\text{off})^2 + p^2 - 2\Delta V - \omega^2 (q_\text{jump}-q_\text{off})^2} \\
   	Q &= q_\text{off}+\cos(\omega (h-t))(q_\text{jump}-q_\text{off})+\sin(\omega (h-t))\bar{p}/\omega \\
   	P &= -\omega\sin(\omega (h-t))(q_\text{jump}-q_\text{off})+\cos(\omega (h-t)) \bar{p}.
\end{align*}
The numerical solution produced by Integrator \ref{int_1stOrderSymplectic}, denoted by $Q_1,P_1$, is given by
\begin{align*}
	\hat{p}_1 &= p - h \omega^2/2 (q-q_\text{off}) \\
	\tau &= (q_\text{jump}-q)/\hat{p}_1 \\
	\hat{p}_2 &= \sqrt{\hat{p}_1^2 - 2\Delta V} \\
	Q_1 &= q_\text{jump}+(h-\tau) \hat{p}_2 \\
	P_1 &= \hat{p}_2 - h\omega^2/2 (Q_1-q_\text{off})
\end{align*}

\paragraph{Position.} Its truncation error is only 2nd-order. More precisely,
we check how well $Q_1$ approximates $Q$ by letting
\[
	a_0 = \lim_{h\rightarrow 0} (Q-Q_1), \quad a_1 = \lim_{h\rightarrow 0} \frac{Q-Q_1-a_0}{h}, \quad a_2 = \lim_{h\rightarrow 0} \frac{Q-Q_1-a_0-a_1 h}{h^2}.
\]
Laborious algebra will show that
\[
	a_0 = 0, \quad a_1 = 0, \quad a_2 = \frac{2 C \Delta V+ (C p-p^2)\left(p-\sqrt{p^2-2\Delta V}\right)}{2p^3\sqrt{p^2-2\Delta V}} (C-p) (q_\text{off}-q_\text{jump}) \omega^2,
\]
which means $Q=Q_1+\mathcal{O}(h^2)$. However, if $\Delta V=0$, it can be checked that $a_2=0$, which means the truncation error returns to be 3rd-order, and that is consistent with the fact that the integrator should be 2nd-order in the smooth case.

\paragraph{Momentum.} Its truncation error is only 1st-order. More precisely,
we check how well $P_1$ approximates $P$ by letting
\[
	b_0 = \lim_{h\rightarrow 0} (P-P_1), \quad b_1 = \lim_{h\rightarrow 0} \frac{P-P_1-b_0}{h}, \quad b_2 = \lim_{h\rightarrow 0} \frac{P-P_1-b_0-b_1 h}{h^2}.
\]
Laborious algebra will show that
\begin{align*}
	& b_0 = 0, \quad b_1 = \frac{p-\sqrt{p^2-2\Delta V}}{2p\sqrt{p^2-2\Delta V}} (2C-p) (q_\text{off}-q_\text{jump}) \omega^2, \\
	& b_2 = \frac{\omega^2}{4 p^3 \alpha^4} \left(-2 C^2 \left(4 \Delta V^2 \left(p \alpha-\beta\right)-2 \Delta V p^2 \left(p \alpha-2 \beta\right)+p^3 \left(\alpha-p\right) \beta\right)-4 C \Delta V p^2 \alpha^3 +\Delta V p^3 \alpha \beta\right),
\end{align*}
where $\alpha=\sqrt{p^2-2 \Delta V}$ and $\beta=\omega ^2 (q_\text{jump}-q_\text{off})^2$.

This means $P=P_1+\mathcal{O}(h)$. However, if $\Delta V=0$, it can be checked that $b_1=0$ and $b_2=0$, which means the truncation error returns to be 3rd-order. That is again consistent with the 2nd-order nature in the smooth case.

\bibliographystyle{siam}
\bibliography{molei27}

\def\cprime{$'$} \def\cprime{$'$} \def\cprime{$'$}
\begin{thebibliography}{10}

\bibitem{aarseth1963dynamical}
{\sc S.~J. Aarseth and F.~Hoyle}, {\em Dynamical evolution of clusters of
  galaxies, i}, Monthly Notices of the Royal Astronomical Society, 126 (1963),
  pp.~223--255.

\bibitem{Ambrosio}
{\sc L.~Ambrosio}, {\em Transport equation and cauchy problem for bv vector
  fields}, Inventiones mathematicae, 158 (2004), pp.~227--260.

\bibitem{benettin1994hamiltonian}
{\sc G.~Benettin and A.~Giorgilli}, {\em On the hamiltonian interpolation of
  near-to-the identity symplectic mappings with application to symplectic
  integration algorithms}, Journal of Statistical Physics, 74 (1994),
  pp.~1117--1143.

\bibitem{blanes2017concise}
{\sc S.~Blanes and F.~Casas}, {\em A concise introduction to geometric
  numerical integration}, CRC press, 2017.

\bibitem{blanes2013optimized}
{\sc S.~Blanes, F.~Casas, P.~Chartier, and A.~Murua}, {\em Optimized high-order
  splitting methods for some classes of parabolic equations}, Mathematics of
  Computation, 82 (2013), pp.~1559--1576.

\bibitem{blanes2013new}
{\sc S.~Blanes, F.~Casas, A.~Farres, J.~Laskar, J.~Makazaga, and A.~Murua},
  {\em New families of symplectic splitting methods for numerical integration
  in dynamical astronomy}, Applied Numerical Mathematics, 68 (2013),
  pp.~58--72.

\bibitem{bond2008stabilized}
{\sc S.~D. Bond and B.~J. Leimkuhler}, {\em Stabilized integration of
  hamiltonian systems with hard-sphere inequality constraints}, SIAM Journal on
  Scientific Computing, 30 (2008), pp.~134--147.

\bibitem{MR1436164}
{\sc F.~A. Bornemann and C.~Sch{\"u}tte}, {\em Homogenization of {H}amiltonian
  systems with a strong constraining potential}, Phys. D, 102 (1997),
  pp.~57--77.

\bibitem{bunimovich2001mushrooms}
{\sc L.~A. Bunimovich}, {\em Mushrooms and other billiards with divided phase
  space}, Chaos: An Interdisciplinary Journal of Nonlinear Science, 11 (2001),
  pp.~802--808.

\bibitem{calvo1995accurate}
{\sc M.-P. Calvo and E.~Hairer}, {\em Accurate long-term integration of
  dynamical systems}, Appl. Numer. Math., 18 (1995), pp.~95--105.

\bibitem{calvo1993development}
{\sc M.~P. Calvo and J.~Sanz-Serna}, {\em The development of variable-step
  symplectic integrators, with application to the two-body problem}, SIAM
  Journal on Scientific Computing, 14 (1993), pp.~936--952.

\bibitem{castella2009splitting}
{\sc F.~Castella, P.~Chartier, S.~Descombes, and G.~Vilmart}, {\em Splitting
  methods with complex times for parabolic equations}, BIT Numerical
  Mathematics, 49 (2009), pp.~487--508.

\bibitem{cirak2005decomposition}
{\sc F.~Cirak and M.~West}, {\em Decomposition contact response (dcr) for
  explicit finite element dynamics}, International Journal for Numerical
  Methods in Engineering, 64 (2005), pp.~1078--1110.

\bibitem{MR2275175}
{\sc D.~Cohen, T.~Jahnke, K.~Lorenz, and C.~Lubich}, {\em Numerical integrators
  for highly oscillatory {H}amiltonian systems: a review}, in Analysis,
  modeling and simulation of multiscale problems, Springer, Berlin, 2006,
  pp.~553--576.

\bibitem{creutz1989higher}
{\sc M.~Creutz and A.~Gocksch}, {\em Higher-order hybrid {M}onte {C}arlo
  algorithms}, Physical Review Letters, 63 (1989), p.~9.

\bibitem{deuflhard2008contact}
{\sc P.~Deuflhard, R.~Krause, and S.~Ertel}, {\em A contact-stabilized newmark
  method for dynamical contact problems}, International Journal for Numerical
  Methods in Engineering, 73 (2008), pp.~1274--1290.

\bibitem{dharmaraja2012time}
{\sc S.~Dharmaraja, H.~Kesari, E.~Darve, and A.~J. Lew}, {\em Time integrators
  based on approximate discontinuous hamiltonians}, International journal for
  numerical methods in engineering, 89 (2012), pp.~71--104.

\bibitem{dieci2009sliding}
{\sc L.~Dieci and L.~Lopez}, {\em Sliding motion in filippov differential
  systems: theoretical results and a computational approach}, SIAM Journal on
  Numerical Analysis, 47 (2009), pp.~2023--2051.

\bibitem{dijkstra2002phase}
{\sc M.~Dijkstra}, {\em Phase behavior of hard spheres with a short-range
  yukawa attraction}, Physical Review E, 66 (2002), p.~021402.

\bibitem{DiPernaLions}
{\sc R.~J. DiPerna and P.-L. Lions}, {\em Ordinary differential equations,
  transport theory and sobolev spaces}, Inventiones mathematicae, 98 (1989),
  pp.~511--547.

\bibitem{doyen2011time}
{\sc D.~Doyen, A.~Ern, and S.~Piperno}, {\em Time-integration schemes for the
  finite element dynamic {S}ignorini problem}, SIAM Journal on Scientific
  Computing, 33 (2011), pp.~223--249.

\bibitem{feng1986difference}
{\sc K.~Feng}, {\em Difference schemes for {H}amiltonian formalism and
  symplectic geometry}, J. Comput. Math., 4 (1986), pp.~279--289.

\bibitem{feng2010symplectic}
{\sc K.~Feng and M.~Qin}, {\em Symplectic geometric algorithms for Hamiltonian
  systems}, Springer, 2010.

\bibitem{fetecau2003nonsmooth}
{\sc R.~C. Fetecau, J.~E. Marsden, M.~Ortiz, and M.~West}, {\em Nonsmooth
  lagrangian mechanics and variational collision integrators}, SIAM Journal on
  Applied Dynamical Systems, 2 (2003), pp.~381--416.

\bibitem{filippov2013differential}
{\sc A.~F. Filippov}, {\em Differential equations with discontinuous righthand
  sides: control systems}, vol.~18, Springer Science \& Business Media, 2013.

\bibitem{forest1989canonical}
{\sc E.~Forest}, {\em Canonical integrators as tracking codes}, in AIP Conf.
  Proc., vol.~184, AIP Publishing, 1989, pp.~1106--1136.

\bibitem{fridman2002slow}
{\sc L.~M. Fridman}, {\em Slow periodic motions with internal sliding modes in
  variable structure systems}, International Journal of Control, 75 (2002),
  pp.~524--537.

\bibitem{Skeel:99}
{\sc B.~Garc\'{i}a-Archilla, J.~M. Sanz-Serna, and R.~D. Skeel}, {\em
  Long-time-step methods for oscillatory differential equations}, SIAM J. Sci.
  Comput., 20 (1999), pp.~930--963.

\bibitem{gerber1996global}
{\sc R.~A. Gerber}, {\em Global effects of softening n-body galaxies}, The
  Astrophysical Journal, 466 (1996), p.~724.

\bibitem{gonthier2004regularized}
{\sc Y.~Gonthier, J.~McPhee, C.~Lange, and J.-C. Piedboeuf}, {\em A regularized
  contact model with asymmetric damping and dwell-time dependent friction},
  Multibody System Dynamics, 11 (2004), pp.~209--233.

\bibitem{grizzle2014models}
{\sc J.~W. Grizzle, C.~Chevallereau, R.~W. Sinnet, and A.~D. Ames}, {\em
  Models, feedback control, and open problems of 3d bipedal robotic walking},
  Automatica, 50 (2014), pp.~1955--1988.

\bibitem{hairer1994backward}
{\sc E.~Hairer}, {\em Backward analysis of numerical integrators and symplectic
  methods}, Annals of Numerical Mathematics, 1 (1994), pp.~107--132.

\bibitem{hairer2000long}
{\sc E.~Hairer and C.~Lubich}, {\em Long-time energy conservation of numerical
  methods for oscillatory differential equations}, SIAM journal on numerical
  analysis, 38 (2000), pp.~414--441.

\bibitem{Hairer06}
{\sc E.~Hairer, C.~Lubich, and G.~Wanner}, {\em Geometric Numerical
  Integration: Structure-Preserving Algorithms for Ordinary Differential
  Equations}, Springer, Berlin Heidelberg New York, second~ed., 2006.

\bibitem{hansen2009high}
{\sc E.~Hansen and A.~Ostermann}, {\em High order splitting methods for
  analytic semigroups exist}, BIT Numerical Mathematics, 49 (2009),
  pp.~527--542.

\bibitem{heyes1982molecular}
{\sc D.~Heyes}, {\em Molecular dynamics simulations of restricted primitive
  model 1: 1 electrolytes}, Chemical Physics, 69 (1982), pp.~155--163.

\bibitem{higham1993accuracy}
{\sc N.~J. Higham}, {\em The accuracy of floating point summation}, SIAM
  Journal on Scientific Computing, 14 (1993), pp.~783--799.

\bibitem{houndonougbo2000molecular}
{\sc Y.~A. Houndonougbo, B.~B. Laird, and B.~J. Leimkuhler}, {\em A molecular
  dynamics algorithm for mixed hard-core/continuous potentials}, Molecular
  Physics, 98 (2000), pp.~309--316.

\bibitem{jackson2007classical}
{\sc J.~D. Jackson}, {\em Classical electrodynamics}, John Wiley \& Sons, 2007.

\bibitem{JinReview}
{\sc S.~Jin}, {\em Numerical methods for hyperbolic systems with singular
  coefficients: well-balanced scheme, hamiltonian preservation, and beyond},
  Hyperbolic Problems: Theory, Numerics and Applications, 67 (2009), p.~93.

\bibitem{Jin-Liao}
{\sc S.~Jin and X.~Liao}, {\em A hamiltonian-preserving scheme for high
  frequency elastic waves in heterogeneous media}, Journal of Hyperbolic
  Differential Equations, 3 (2006), pp.~741--777.

\bibitem{JinNovak1}
{\sc S.~Jin and K.~A. Novak}, {\em A semiclassical transport model for thin
  quantum barriers}, Multiscale Modeling \& Simulation, 5 (2006),
  pp.~1063--1086.

\bibitem{JinNovak2}
\leavevmode\vrule height 2pt depth -1.6pt width 23pt, {\em A semiclassical
  transport model for two-dimensional thin quantum barriers}, Journal of
  Computational Physics, 226 (2007), pp.~1623--1644.

\bibitem{JinQiZhang}
{\sc S.~Jin, P.~Qi, and Z.~Zhang}, {\em An eulerian surface hopping method for
  the schr{\"o}dinger equation with conical crossings}, Multiscale Modeling \&
  Simulation, 9 (2011), pp.~258--281.

\bibitem{Jin-Wen1}
{\sc S.~Jin and X.~Wen}, {\em Hamiltonian-preserving schemes for the liouville
  equation with discontinuous potentials}, Communications in Mathematical
  Sciences, 3 (2005), pp.~285--315.

\bibitem{Jin-Wen3}
\leavevmode\vrule height 2pt depth -1.6pt width 23pt, {\em A
  hamiltonian-preserving scheme for the liouville equation of geometrical
  optics with partial transmissions and reflections}, SIAM Journal on Numerical
  Analysis, 44 (2006), pp.~1801--1828.

\bibitem{Jin-Wen2}
\leavevmode\vrule height 2pt depth -1.6pt width 23pt, {\em
  Hamiltonian-preserving schemes for the liouville equation of geometrical
  optics with discontinuous local wave speeds}, Journal of Computational
  Physics, 214 (2006), pp.~672--697.

\bibitem{JinWuHuang}
{\sc S.~Jin, H.~Wu, and Z.~Huang}, {\em A hybrid phase-flow method for
  hamiltonian systems with discontinuous hamiltonians}, SIAM Journal on
  Scientific Computing, 31 (2009), pp.~1303--1321.

\bibitem{Jin-Yin}
{\sc S.~Jin and D.~Yin}, {\em Computational high frequency waves through curved
  interfaces via the liouville equation and geometric theory of diffraction},
  Journal of Computational Physics, 227 (2008), pp.~6106--6139.

\bibitem{kane1999finite}
{\sc C.~Kane, E.~A. Repetto, M.~Ortiz, and J.~E. Marsden}, {\em Finite element
  analysis of nonsmooth contact}, Computer methods in applied mechanics and
  engineering, 180 (1999), pp.~1--26.

\bibitem{kaufman2012geometric}
{\sc D.~M. Kaufman and D.~K. Pai}, {\em Geometric numerical integration of
  inequality constrained, nonsmooth hamiltonian systems}, SIAM Journal on
  Scientific Computing, 34 (2012), pp.~A2670--A2703.

\bibitem{khenous2008mass}
{\sc H.~B. Khenous, P.~Laborde, and Y.~Renard}, {\em Mass redistribution method
  for finite element contact problems in elastodynamics}, European Journal of
  Mechanics-A/Solids, 27 (2008), pp.~918--932.

\bibitem{krause2012presentation}
{\sc R.~Krause and M.~Walloth}, {\em Presentation and comparison of selected
  algorithms for dynamic contact based on the newmark scheme}, Applied
  Numerical Mathematics, 62 (2012), pp.~1393--1410.

\bibitem{laursen1997design}
{\sc T.~Laursen and V.~Chawla}, {\em Design of energy conserving algorithms for
  frictionless dynamic contact problems}, International Journal for Numerical
  Methods in Engineering, 40 (1997), pp.~863--886.

\bibitem{laursen2002improved}
{\sc T.~Laursen and G.~Love}, {\em Improved implicit integrators for transient
  impact problems—geometric admissibility within the conserving framework},
  International Journal for Numerical Methods in Engineering, 53 (2002),
  pp.~245--274.

\bibitem{leimkuhler2004simulating}
{\sc B.~Leimkuhler and S.~Reich}, {\em Simulating {H}amiltonian dynamics},
  vol.~14, Cambridge University Press, 2004.

\bibitem{leine2013dynamics}
{\sc R.~I. Leine and H.~Nijmeijer}, {\em Dynamics and bifurcations of
  non-smooth mechanical systems}, vol.~18, Springer Science \& Business Media,
  2013.

\bibitem{leyendecker2012variational}
{\sc S.~Leyendecker, C.~Hartmann, and M.~Koch}, {\em Variational collision
  integrator for polymer chains}, Journal of Computational Physics, 231 (2012),
  pp.~3896--3911.

\bibitem{llibre2007regularization}
{\sc J.~Llibre, P.~R. Da~Silva, and M.~A. Teixeira}, {\em Regularization of
  discontinuous vector fields on $\mathbb{R}^3$ via singular perturbation},
  Journal of Dynamics and Differential Equations, 19 (2007), pp.~309--331.

\bibitem{llibre2008sliding}
{\sc J.~Llibre, P.~R. Da~Silva, M.~A. Teixeira, et~al.}, {\em Sliding vector
  fields via slow--fast systems}, Bulletin of the Belgian Mathematical
  Society-Simon Stevin, 15 (2008), pp.~851--869.

\bibitem{makarenkov2012dynamics}
{\sc O.~Makarenkov and J.~S. Lamb}, {\em Dynamics and bifurcations of nonsmooth
  systems: A survey}, Physica D: Nonlinear Phenomena, 241 (2012),
  pp.~1826--1844.

\bibitem{MaWe:01}
{\sc J.~E. Marsden and M.~West}, {\em Discrete mechanics and variational
  integrators}, Acta Numer., 10 (2001), pp.~357--514.

\bibitem{McQu02}
{\sc R.~McLachlan and G.~R.~W. Quispel}, {\em Splitting methods}, Acta Numer.,
  (2002), pp.~341--434.

\bibitem{mcneil1982new}
{\sc W.~J. McNeil and W.~G. Madden}, {\em A new method for the molecular
  dynamics simulation of hard core molecules}, The Journal of Chemical Physics,
  76 (1982), pp.~6221--6226.

\bibitem{moser1968lectures}
{\sc J.~Moser}, {\em Lectures on {H}amiltonian systems}, vol.~81, American
  Mathematical Soc., 1968.

\bibitem{Noether:18}
{\sc E.~Noether}, {\em Invariante variationsprobleme}, Nachr. D. K{\"o}nig.
  Gesellsch. D. Wiss. Zu G{\"o}ttingen, Math-phys. Klasse,  (1918),
  pp.~235--257.

\bibitem{nordmark2011friction}
{\sc A.~Nordmark, H.~Dankowicz, and A.~Champneys}, {\em Friction-induced
  reverse chatter in rigid-body mechanisms with impacts}, IMA Journal of
  Applied Mathematics, 76 (2011), pp.~85--119.

\bibitem{pandolfi2002time}
{\sc A.~Pandolfi, C.~Kane, J.~E. Marsden, and M.~Ortiz}, {\em Time-discretized
  variational formulation of non-smooth frictional contact}, International
  Journal for Numerical Methods in Engineering, 53 (2002), pp.~1801--1829.

\bibitem{pekarek2012variational}
{\sc D.~Pekarek and T.~D. Murphey}, {\em Variational nonsmooth mechanics via a
  projected hamilton's principle}, in 2012 American Control Conference (ACC),
  IEEE, 2012, pp.~1040--1046.

\bibitem{PerthameSimeoni}
{\sc B.~Perthame and C.~Simeoni}, {\em A kinetic scheme for the saint-venant
  system with a source term}, Calcolo, 38 (2001), pp.~201--231.

\bibitem{quispel1998volume}
{\sc G.~Quispel and C.~Dyt}, {\em Volume-preserving integrators have linear
  error growth}, Phys. Lett. A, 242 (1998), pp.~25--30.

\bibitem{calvo1994numerical}
{\sc J.~Sanz-Serna and M.~Calvo}, {\em Numerical {H}amiltonian problems},
  Chapman and Hall/CRC, 1st~ed., 1994.

\bibitem{sanz1992symplectic}
{\sc J.~M. Sanz-Serna}, {\em Symplectic integrators for {H}amiltonian problems:
  an overview}, Acta Numer., 1 (1992), pp.~243--286.

\bibitem{Sanz-Serna:08}
{\sc J.~M. Sanz-Serna}, {\em Mollified impulse methods for highly oscillatory
  differential equations}, SIAM J. Numer. Anal., 46 (2) (2008), pp.~1040--1059.

\bibitem{SoTa18}
{\sc A.~Souza and M.~Tao}, {\em Metastable transitions in inertial {L}angevin
  systems: what can be different from the overdamped case?}, European Journal
  of Applied Mathematics,  (2018).

\bibitem{stewart2000rigid}
{\sc D.~E. Stewart}, {\em Rigid-body dynamics with friction and impact}, SIAM
  review, 42 (2000), pp.~3--39.

\bibitem{Strang:68}
{\sc G.~Strang}, {\em On the construction and comparison of difference
  schemes}, SIAM J. Numer. Anal., 5 (1968), pp.~506--517.

\bibitem{stratt1981constrained}
{\sc R.~M. Stratt, S.~L. Holmgren, and D.~Chandler}, {\em Constrained impulsive
  molecular dynamics}, Molecular Physics, 42 (1981), pp.~1233--1143.

\bibitem{suh1990molecular}
{\sc S.-H. Suh, L.~Mier-y Teran, H.~White, and H.~Davis}, {\em Molecular
  dynamics study of the primitive model of 1--3 electrolyte solutions},
  Chemical physics, 142 (1990), pp.~203--211.

\bibitem{suzuki1990fractal}
{\sc M.~Suzuki}, {\em Fractal decomposition of exponential operators with
  applications to many-body theories and {M}onte {C}arlo simulations}, Phys.
  Lett. A, 146 (1990), pp.~319--323.

\bibitem{Ta16}
{\sc M.~Tao}, {\em Explicit high-order symplectic integrators for charged
  particles in general electromagnetic fields}, Journal of Computational
  Physics, 327 (2016), pp.~245--251.

\bibitem{Tao2016PRE}
\leavevmode\vrule height 2pt depth -1.6pt width 23pt, {\em Explicit symplectic
  approximation of nonseparable {H}amiltonians: algorithm and long time
  performance}, Phys. Rev. E, 94 (2016), p.~043303.

\bibitem{tao2016variational}
{\sc M.~Tao and H.~Owhadi}, {\em Variational and linearly-implicit integrators,
  with applications}, IMA Journal of Numerical Analysis, 36 (2016),
  pp.~80--107.

\bibitem{FLAVOR10}
{\sc M.~Tao, H.~Owhadi, and J.~E. Marsden}, {\em Nonintrusive and structure
  preserving multiscale integration of stiff {ODEs}, {SDEs} and {Hamiltonian}
  systems with hidden slow dynamics via flow averaging}, Multiscale Modeling \&
  Simulation: A SIAM Interdisciplinary Journal, 8 (2010), pp.~1269--1324.

\bibitem{SIM2}
\leavevmode\vrule height 2pt depth -1.6pt width 23pt, {\em From efficient
  symplectic exponentiation of matrices to symplectic integration of
  high-dimensional {H}amiltonian systems with slowly varying quadratic stiff
  potentials}, Appl. Math. Res. Express,  (2011), pp.~242--280.

\bibitem{teixeira2012regularization}
{\sc M.~A. Teixeira and P.~R. da~Silva}, {\em Regularization and singular
  perturbation techniques for non-smooth systems}, Physica D: Nonlinear
  Phenomena, 241 (2012), pp.~1948--1955.

\bibitem{Yoshida:90}
{\sc H.~Yoshida}, {\em Construction of higher order symplectic integrators},
  Phys. Lett. A, 150 (1990), pp.~262--268.

\end{thebibliography}

\end{document}